\documentclass{amsart}
\usepackage{url}
\usepackage{hyperref}
\hypersetup{colorlinks=true, citecolor=blue}
\providecommand{\texorpdfstring}[2]{#1}
\usepackage[lite]{amsrefs}
\usepackage{amssymb}
\usepackage{mathrsfs}
\usepackage[shortlabels]{enumitem}
\usepackage[all,cmtip]{xy}

\newcommand{\tensor}{\otimes}
\newcommand{\isom}{\cong}

\newcommand{\C}{\mathbb{C}}

\newcommand{\Q}{\mathbb{Q}}
\newcommand{\Z}{\mathbb{Z}}
\newcommand{\A}{\mathbb{A}}
\renewcommand{\P}{\mathbb{P}}

\newcommand{\M}{\mathcal{M}}

\newcommand{\F}{\mathcal{F}}

\renewcommand{\O}{\mathcal{O}}

\renewcommand{\L}{\mathcal{L}}

\newcommand{\codim}{\textrm{codim}}
\newcommand{\w}{\omega}
\renewcommand{\k}{\kappa}
\newcommand{\K}{\mathcal{K}}
\newcommand{\gr}{\textrm{gr}}

\newcommand{\Koz}{\textrm{Koz}}

\numberwithin{equation}{section}

\theoremstyle{plain}
\newtheorem{thm}[equation]{Theorem}
\newtheorem{cor}[equation]{Corollary}
\newtheorem{lem}[equation]{Lemma}
\newtheorem{prop}[equation]{Proposition}
\newtheorem{conj}[equation]{Conjecture}

\theoremstyle{definition}
\newtheorem{defn}[equation]{Definition}

\theoremstyle{remark}
\newtheorem{rmk}[equation]{Remark}
\newtheorem{ex}[equation]{Example}

\begin{document}

\title[Logarithmic base change theorem]{Logarithmic base change theorem and smooth descent of positivity of log canonical divisor}
\author{Sung Gi Park}
\address{Department of Mathematics, Harvard University, 1 Oxford Street, Cambridge, MA 02138, USA}
\email{sgpark@math.harvard.edu}

\date{\today}

\begin{abstract}
We prove a logarithmic base change theorem for pushforwards of pluri-canonical bundles and use it to deduce that positivity properties of log canonical divisors descend via smooth projective morphisms.

As an application, for a surjective morphism \(f:X\to Y\) with \(\k(X)\ge 0\) and \(-K_Y\) big,  we prove \(Y\setminus \Delta(f)\) is of log general type, where \(\Delta(f)\) is the discriminant locus. In particular, when \(Y=\P^n\) we have \(\dim \Delta(f)=n-1\) and \(\deg \Delta(f)\ge n+2\), generalizing the case \(n=1\) proved by Viehweg-Zuo. In addition, we prove Popa's conjecture on the superadditivity of the logarithmic Kodaira dimension of smooth algebraic fiber spaces over bases of dimension at most three and analyze related problems.
\end{abstract}

\maketitle

\tableofcontents

\section{Introduction}
\label{sec:intro}

In this paper, we prove a number of results on the behavior of positivity, hyperbolicity, and Kodaira dimension under smooth morphisms of quasi-projective varieties, based on a technical tool we study under the name of logarithmic base change theorem. Throughout the text, a variety is a reduced separated scheme of finite type over \(\C\). A pair \((Y,D)\) consists of a variety \(Y\) and a formal \(\Q\)-linear combination of divisors \(D\), and we call it a log smooth pair if \(Y\) is smooth and \(D\) is a reduced simple normal crossing divisor.

Given a surjective morphism \(f:X\to Y\) of smooth projective varieties, the study of the discriminant locus \(\Delta(f)\subset Y\), defined as the set of points \(y\in Y\) with singular scheme-theoretic fibers \(X_y:=f^*(y)\), has been an active area of research. For example, Catanese-Schneider \cite{CS95}*{Question 4.1} asked the question whether \(f:X\to \P^1\) has at least \(3\) singular fibers if \(X\) is of general type. Kov\'acs \cite{Kovacs00} gave a partial answer when \(X\) is canonically polarized and Viehweg-Zuo \cite{VZ01} answered affirmatively under the weaker assumption that \(X\) has non-negative Kodaira dimension. The question is alternatively phrased as follows: the effectivity of the canonical divisor of \(X\) imposes the bigness of the log canonical divisor of \(\P^1\setminus \Delta(f)\).

We refine and generalize this phenomenon to higher dimensional bases; given a smooth projective morphism \(f:X\to Y\) of smooth quasi-projective varieties, the positivity of the log canonical divisor of \(X\) descends to the positivity of the log canonical divisor of \(Y\). For example, bigness (resp. effectivity) descends to bigness (resp. pseudo-effectivity). The precise statements are obtained as consequences of our main technical result, which we call the \textit{logarithmic base change theorem} for pushforwards of logarithmic pluri-canonical bundles.

A priori, Viehweg introduced the base change theorem for pushforwards of pluri-canonical bundles to study the subadditivity of the Kodaira dimension for algebraic fiber spaces. Here is a version given by Mori \cite{Mori87}*{(4.10)}.

\begin{thm}[\cites{Viehweg83, Mori87}, Base Change Theorem]
\label{thm: base change theorem}
Let \(X, Y, Y'\) be smooth quasi-projective varieties and \(f:X\to Y\) be a surjective projective morphism and \(g:Y'\to Y\) be a flat projective morphism. Let \(\mu:X''\to X'=X\times_Y Y'\) be a resolution of singularities. Consider the following commutative diagram:
\begin{displaymath}
\xymatrix{
{X}\ar[d]_-{f}& {X'}\ar[l]_-{g'}\ar[d]_-{f'}& {X''}\ar[l]_-{\mu}\ar[ld]^-{f''}\\
{Y}&{Y'}\ar[l]^-{g}&\\
}
\end{displaymath}
Then for a positive integer \(N\),
\begin{enumerate}
    \item there is an inclusion
    \[
    f''_*\omega^N_{X''/Y'}\subset g^*f_*\omega_{X/Y}^N.
    \]
    It is an equality at a point \(y'\in Y'\) if \(g(y')\) is a codimension 1 point of \(Y\) and \(f\) is semistable in the neighborhood of \(g(y')\).
    \item there is an inclusion
    \[
    g_*f''_*\omega^N_{X''/Y}\subset (f_* \omega^N_{X/Y}\tensor g_*\omega_{Y'/Y}^N)^{**}.
    \]
    It is an equality at a codimension \(1\) point \(y\in Y\) if \(f\) or \(g\) is semistable in the neighborhood of \(y\), where \(^{**}\) denotes the double dual.
\end{enumerate}
\end{thm}

Note that the morphism \(f\) is semistable over a codimension \(1\) point \(y\in Y\) if the fiber \(X_y\) has reduced normal crossing singularities. Outside of a codimension \(2\) subvariety, Viehweg's base change theorem gives an inclusion of pushforwards of pluri-canonical bundles of the fiber product to the tensor products of pushforwards of pluri-canonical bundles of each morphism.

We start by pointing out a rather striking new phenomenon: by suitably introducing poles in the discriminant loci of morphisms, the inclusion in the base change theorem is reversed. This represents a distinct departure from previously observed phenomena.

\begin{thm}[Logarithmic Base Change Theorem]
\label{thm: logarithmic base change theorem}
Let \((X,E), (Y,D), (Y',D')\) be quasi-projective log smooth pairs, and let \(f:(X,E)\to (Y,D)\), \(g:(Y',D')\to (Y,D)\) be surjective projective morphisms of pairs such that \(E=f^{-1}(D), D'=g^{-1}(D)\) and \(g|_{Y'\setminus D'}:Y'\setminus D'\to Y\setminus D\) is smooth. Let \(X'\) be the union of the irreducible components of \(X\times_YY'\) dominating \(Y\), and \(E'=g'^{-1}(E)\). Consider the following commutative diagram:
\begin{equation}
\label{log diagram}
\xymatrix{
{(X,E)}\ar[d]_-{f}& {(X',E')}\ar[l]_-{g'}\ar[d]_-{f'}& {(X'',E'')}\ar[l]_-{\mu}\ar[ld]^-{f''}\\
{(Y,D)}&{(Y',D')}\ar[l]^-{g}&\\
}
\end{equation}
where \(\mu:(X'',E'')\to (X',E')\) is a log resolution of pairs with \(X''\setminus E''\isom X'\setminus E'\), so that \((g\circ f'')^{-1}D=E''\). Then there exists an inclusion for a positive integer \(N\):
\[
\left[f_* \left(\w_X(E)/\w_Y(D)^{\tensor N}\right)\tensor g_*\left(\w_{Y'}(D')/\w_Y(D)^{\tensor N}\right)\right]^{**}\subset\left[h_*\left(\omega_{X''}(E'')/\omega_Y(D)^{\tensor N}\right)\right]^{**}
\]
where \(h=g\circ f''\).
\end{thm}

Here, \(\w_X(E)/\w_Y(D):=\w_X(E)\tensor f^*\w_Y(D)^{-1}\) and \(\w_X(E)/\w_Y(D)^{\tensor N}\) is the \(N\)-tensor power of this line bundle. Moreover, \(f^{-1}(\bullet)\) indicates the set-theoretic preimage of \(\bullet\) with the reduced scheme structure. While the equality condition in Viehweg's base change theorem follows from the observation that the fiber product \(X'\) has normal toric singularities, the fiber product \(X'\) in the logarithmic base change theorem has binomial hypersurface singularities. The proof of Theorem \ref{thm: logarithmic base change theorem} follows from a careful analysis of Bierstone and Milman's resolution of binomial hypersurface singularities \cite{BM06} applied to singularities of pairs.

Viehweg applied the base change theorem inductively on the iterated fiber products to study Iitaka's \(C_{n,m}^+\) conjecture on the subadditivity of the Kodaira dimension; this technique is the so-called Viehweg's fiber product trick. Likewise, we employ a logarithmic analogue of this fiber product trick, which is described in the following paragraph.

Let \(f:(X,E)\to (Y,D)\) be a projective morphism of quasi-projective log smooth pairs such that \(E=f^{-1}(D)\) and \(f|_{X\setminus E}\) is smooth. Define \(X^s:=X\times_Y\cdots\times_YX\) as the \(s\)-fold fiber product of \(f\), with the induced morphism \(f^s:X^s\to Y\), and define \(E^s=(f^s)^{-1}(D)\). Then \(f^s|_{X^s\setminus E^s}\) is a smooth projective morphism. Choose a log resolution \(\mu^s:(X^{(s)},E^{(s)}) \to (X^s,E^s)\) of the union of the irreducible components of \(X^s\) dominating \(Y\), with the induced morphism \(f^{(s)}:(X^{(s)},E^{(s)})\to (Y,D)\) of log smooth pairs satisfying \(E^{(s)}=(f^{(s)})^{-1}(D)\) and \(f^{(s)}|_{X^{(s)}\setminus E^{(s)}}\) smooth. Then we have:

\begin{cor}[Logarithmic Fiber Product Trick]
\label{cor: logarithmic fiber product trick}
With the notation in the previous paragraph, we have the following inclusion:
\[
\left[\bigotimes^s f_* \left(\w_X(E)/\w_Y(D)^{\tensor N}\right)\right]^{**}\hookrightarrow \left[f^{(s)}_*\left(\w_{X^{(s)}}(E^{(s)})/\w_Y(D)^{\tensor N}\right)\right]^{**}
\]
for \(N,s>0\). 
\end{cor}

Notice that the right-hand side of the above inclusion is independent of the choice of a log resolution of the pair \((X^s,E^s)\). The proof is immediate from Theorem \ref{thm: logarithmic base change theorem} iterated \(s\)-times.

In contrast to Viehweg's base change theorem and fiber product trick, the logarithmic analogues have a distinctive feature: the inclusion goes in the opposite direction, which is consistent with Popa's conjectures \cite{Popa21} on the superadditivity of the logarithmic Kodaira dimension. Specifically, the (log) Kodaira dimension of the source imposes a lower bound on the (log) Kodaira dimension of the base. This phenomenon, namely a \textit{smooth projective descent of positivity of log canonical divisor}, is opposite in nature to Iitaka's \(C_{n,m}^+\) conjecture. As a consequence of the logarithmic base change theorem, we derive a number of results on the (log) Kodaira dimension and the discriminant loci of morphisms.

To begin with, we recall the definition of the logarithmic Kodaira dimension of a quasi-projective variety. For a \(\Q\)-Cartier divisor \(L\) on a normal projective variety \(X\), \(\k(X,L)\) denotes the Iitaka dimension of \(L\). For instance, \(\k(X):=\k(X,\w_X)\) is the Kodaira dimension of \(X\) when \(X\) is a smooth projective variety. For a smooth quasi-projective variety \(X\), the logarithmic Kodaira dimension is defined as
\[
\bar\k(X):=\k(\bar X, K_{\bar X}+D),
\]
where \(\bar X\) is a compactification of \(X\) with boundary a reduced simple normal crossing divisor \(D\). We say \(X\) is of log general type if \(\bar \k(X)=\dim X\). It is well known that the logarithmic Kodaira dimension is well-defined, independent of the choice of a compactification.

As a first application of the logarithmic base change theorem, we prove a logarithmic analogue of Popa-Schnell \cite{PS22}*{Theorem A}.

\begin{thm}
\label{thm: log general type descends}
Let \(f:X\to Y\) be a smooth projective morphism of smooth quasi-projective varieties whose general fiber \(F\) is connected and \(\kappa(F)\ge 0\). Then \(Y\) is of log general type if and only if \(\bar\k(X)=\k(F)+\dim Y\).
\end{thm}

The only if part of the theorem is Iitaka's logarithmic \(C_{n,m}\)-conjecture when the base is of log general type, proven by Maehara \cite{Maehara86}*{Corollary 2}. Therefore, our contribution is the if part, which can be alternatively phrased as follows - if the source has the maximal log Kodaira dimension attained by the equality of the Easy Addition formula \cite{Mori87}*{Corollary 1.7}, \(\bar\k(X)\le\k(F)+\dim Y\), then the base is of log general type. 

In particular, the bigness of the log canonical divisor descends via smooth projective morphisms.

\begin{cor}
\label{cor: log general type descends}
Let \(f:X\to Y\) be a smooth projective morphism of smooth quasi-projective varieties. If \(X\) is of log general type, then \(Y\) is of log general type.
\end{cor}

Analogously, the next theorem states that the effectivity of the log canonical divisor descends to the pseudo-effectivity of the log canonical divisor via smooth projective morphisms.

\begin{thm}
\label{thm: effectivity descends to pseudo-effectivity}
Let \(f:(X,E)\to (Y,D)\) be a surjective morphism of projective log smooth pairs with \(E=f^{-1}(D)\), and \(f|_{X\setminus E}\) is smooth.
\begin{enumerate}
\item If \(K_X+(1-\epsilon) E\) is \(\Q\)-effective for some \(\epsilon>0\), then \(K_Y+(1-\delta)D\) is pseudo-effective for some \(\delta>0\). In particular, if additionally \(-K_Y+ND\) is big for some non-negative integer \(N\), then \(D\) and \(K_Y+D\) are big.
\item If \(K_X+ E\) is \(\Q\)-effective, then \(K_Y+D\) is pseudo-effective.
\end{enumerate}
\end{thm}

Here, we say a divisor is \(\Q\)-effective if it is \(\Q\)-linearly equivalent to an effective divisor. To avoid repetition, we will refer to \(\Q\)-effective divisors simply as effective divisors from now on. The effectivity of \(K_X+(1-\epsilon) E\) (resp. pseudo-effectivity of \(K_Y+(1-\delta)D\)) for small enough \(\epsilon\) (resp. \(\delta\)) is independent of the choice of a compactification of \(X\setminus E\) (resp. \(Y\setminus D\)). This is also true for the effectivity of \(K_X+E\) (resp. pseudo-effectivity of \(K_Y+D\)). Therefore, Theorem \ref{thm: effectivity descends to pseudo-effectivity} can be stated purely in terms of a smooth projective morphism \(f|_{X\setminus E}:X\setminus E\to Y\setminus D\) of smooth quasi-projective varieties.

The proof of (2) is treated separately at the end of Section \ref{sec: Campana-Paun}, since this follows from the results in the literature without the use of the logarithmic base change theorem. This extends \cite{PS22}*{Proposition G} to the logarithmic setting.

\begin{rmk}
Note that Theorem \ref{thm: effectivity descends to pseudo-effectivity} (1) and the non-vanishing conjecture for klt pairs imply the following statement: if \(K_X+(1-\epsilon)E\) is (pseudo-)effective for some \(\epsilon>0\), then \(K_Y+(1-\delta)D\) is (pseudo-)effective for some \(\delta>0\). When \(X\setminus E\to Y\setminus D\) is finite \'etale, this is easily proven. Indeed, the descent of the effectivity follows from the same technique used to prove the invariance of the logarithmic Kodaira dimension under \'etale covers by Iitaka \cite{Iitaka77}*{Theorem 3}. The descent of the pseudo-effectivity is also immediate from Lemma \ref{lem: weakly positive and big sheaves} (5), explained later. Note that when we replace (pseudo-)effective by big in the statement, we obtain Corollary \ref{cor: log general type descends}.
\end{rmk}

The following theorem is an immediate consequence of Theorem \ref{thm: log general type descends} and \ref{thm: effectivity descends to pseudo-effectivity} (1). As explained at the beginning, the discriminant locus \(\Delta(f)\subset Y\) of a morphism \(f:X\to Y\) between smooth varieties is the complement of the locus in \(Y\) over which \(f\) is smooth.

\begin{thm}
\label{thm: degree of discriminant locus}
Let \(f:X\to Y\) be a surjective morphism of smooth projective varieties with \(\k(X)\ge 0\), and let \(D\) be the divisorial component of the discriminant locus \(\Delta(f)\) in \(Y\). Suppose either
\begin{enumerate}[(i)]
    \item \(\k(X)=\k(F)+\dim Y\) where \(F\) is the general fiber of \(f\), or
    \item \(-K_Y\) is big.
\end{enumerate}
Then, \(Y\setminus\Delta(f)\) is of log general type. In particular, \(K_Y+D\) is big.
\end{thm}

This extends to a normal variety \(Y\) when we extend the notion of bigness to a rank 1 reflexive sheaf. Part (i) extends a result from \cite{PS22}*{Remark 1} which further assumed that \(Y\) is not uniruled. Part (ii) can be seen as a vast generalization of Catanese and Schneider's question \cite{CS95}*{Question 4.1} and Viehweg and Zuo's result \cite{VZ01}*{Theorem 0.2} stating that a surjective morphism \(f:X\to \P^1\) with \(\k(X)\ge 0\) has at least three singular fibers. For instance, when \(Y=\P^n\) we have:

\begin{cor}
\label{cor: degree of discriminant locus}
Let \(f:X\to \P^n\) be a surjective morphism from a smooth projective variety \(X\) of non-negative Kodaira dimension. Then \(\dim\Delta(f)=n-1\) and \(\deg \Delta(f)\ge n+2\).
\end{cor}

\begin{rmk}[{\textbf{Hyper-K\"ahler manifolds}}]
This applies to a Lagrangian fibration \(f:X\to Y\) of a projective hyper-K\"ahler manifold \(X\) of dimension \(2n\). When \(Y\) is smooth, it is known that \(Y=\P^n\) and the discriminant locus \(\Delta(f)\) is a divisor (see \cite{Hwang08}*{Theorem 1.2} and \cite{HO09}*{Proposition 3.1}). Corollary \ref{cor: degree of discriminant locus} implies that its degree is at least \(n+2\) in \(\P^n\).
\end{rmk}

More generally, when \(\P^n\) is replaced by the product of projective spaces \(\P^{n_1}\times\dots\times \P^{n_k}\), then \(\dim\Delta(f)=\sum_{i=1}^k n_i-1\) and the multi-degree \(\deg \Delta(f)=(a_1,\dots, a_k)\) satisfies \(a_i\ge n_i+2\) for all \(i\). Examples \ref{ex: boundary not connected fiber} and \ref{ex: boundary connected fiber} illustrate that the inequality in Corollary \ref{cor: degree of discriminant locus} is sharp.

In our next main application, we prove Popa's conjecture on the superadditivity of the (log) Kodaira dimension for smooth projective morphisms (cf. \cite{Popa21}*{Conjecture 3.1}), under an additional assumption on the base. This assumption is implied by the conjectures of the log minimal model program.

\begin{thm}
\label{thm: superadditivity of logarithmic Kodaira dimension}
Let \(f:X\to Y\) be a smooth projective morphism with connected fibers between smooth quasi-projective varieties. Assume \(\bar \k(Y)\ge 0\), and that the very general fiber of the log Iitaka fibration of \(Y\) has a good minimal model. Then
\[
\bar \k(X)\le\bar \k(Y)+\k(F)
\]
where \(F\) is the general fiber of \(f\).
\end{thm}

Recall that the log Iitaka fibration of \(Y\) is the Iitaka fibration associated to the log canonical bundle \(K_{\bar Y}+D\) of a log smooth pair \((\bar Y, D)\). Here, \(\bar Y\) is a compactification of \(Y\) with boundary a reduced simple normal crossing divisor \(D\). By taking a log resolution \((\widetilde Y, \widetilde D)\) of the pair \((\bar Y, D)\), we have the log Iitaka fibration \(\eta:(\widetilde Y,\widetilde D)\to I\), where the very general fiber \((G, \widetilde D|_G)\) has log Kodaira dimension zero.

When \(\bar \k(Y)=-\infty\), the non-vanishing conjecture for log canonical pairs and Theorem \ref{thm: effectivity descends to pseudo-effectivity} (2) imply \(\bar \k(X)=-\infty\). This conjecture and the existence of good minimal models for log canonical pairs are well known to hold in dimension at most three (see e.g. \cites{KMM94, Shokurov96}). Therefore, we have:

\begin{cor}
\label{cor: superadditivity of logarithmic Kodaira dimension}
Let \(f:X\to Y\) be a smooth projective morphism with connected fibers between smooth quasi-projective varieties. If \(\dim Y\le 3\), then
\[
\bar \k(X)\le\bar \k(Y)+\k(F)
\]
and the equality holds when \(\dim Y=1\).
\end{cor}

This extends \cite{PS22}*{Corollary E and F} which assumed that \(X\) and \(Y\) are projective. On a related note, the logarithmic Iitaka conjecture suggests subadditivity:
\[
\bar\k(X)\ge\bar\k(Y)+\k(F),
\]
even without the smoothness assumption. This is known when \(Y\) is a quasi-projective curve (hence, the equality in Corollary \ref{cor: superadditivity of logarithmic Kodaira dimension}), but not when \(Y\) is an arbitrary quasi-projective surface or a threefold. Theorem \ref{thm: superadditivity of logarithmic Kodaira dimension} and Corollary \ref{cor: superadditivity of logarithmic Kodaira dimension} prove that this inequality is reversed when the morphism is smooth.

\begin{rmk}
In the paper by Campana \cite{Campana22}, the same superadditivity result is phrased when the fibers have semiample canonical bundles. The logarithmic Iitaka conjecture is known to hold when the general fiber \(F\) has a good minimal model \cite{Hashizume20}*{Theorem 1.2}.
\end{rmk}

On a different note, in light of the smooth descent of the positivity of the log canonical divisor, we additionally investigate the uniruledness of \(X\) when \(Y\) is a rational curve.

\begin{thm}
\label{thm: 2 singular fiber uniruled}
Let \(f:X\to \P^1\) be a surjective morphism with connected fibers from a smooth projective variety \(X\). Suppose \(f\) has at most \(2\) singular fibers. If the general fiber \(F\) has a good minimal model, then \(X\) is uniruled. In particular, this holds when \(\k(F)\ge\dim F-3\).\footnote{Here, we use the fact that a smooth projective variety has a good minimal model if the general fiber of its Iitaka fibration has a good minimal model \cite{Lai11}.}
\end{thm}

Indeed, the non-vanishing conjecture and \cite{VZ01} suggest \(X\) is uniruled. Using symplectic geometry, Pieloch \cite{Pieloch21} recently proved that \(X\) is uniruled when \(f\) has at most one singular fiber without the assumption on \(F\), and made some progress when \(f\) has two singular fibers. In general, we do not have an unconditional algebraic proof of these statements.

\subsubsection*{\textnormal{\textbf{What is new.}}}

Viehweg's base change theorem has been one of the keys to studying Iitaka's conjecture and hyperbolicity problems like the Viehweg's conjecture on families with maximal variation; see, for instance, \cites{VZ01, VZ02, KK08, KK10, CP19, PS17}.

The main technical contribution of this paper is the use of the logarithmic base change theorem instead, as a more appropriate tool to answer questions in a logarithmic setting. For instance, the logarithmic fiber product trick allows us to overcome obstructions to generalize Viehweg and Zuo's result as in Corollary \ref{cor: degree of discriminant locus}. Specifically, Viehweg and Zuo's proof for a surjective morphism to \(\P^1\) uses the base change theorem together with semistable reduction (Kempf et al. \cite{KKMS}). This essentially reduces to the case where every fiber is either smooth or a simple normal crossing divisor, so that the equality holds when we apply the base change theorem \ref{thm: base change theorem}. For higher dimensional bases, like \(\P^n\), this reduction procedure does not work properly. However, all technical issues are resolved by the fact that the logarithmic base change theorem reverses inclusions as described above; it relies on a resolution algorithm for binomial hypersurface singularities. For the same reason, the smooth projective descent of the positivity of the log canonical divisor can be verified as mentioned after Corollary \ref{cor: logarithmic fiber product trick}.

\subsubsection*{\textnormal{\textbf{Overview of the paper.}}}

To derive various geometric consequences, we rely on three main technical components, which are as follows:
\begin{enumerate}[(1)]
    \item Logarithmic base change theorem
    \item Construction of logarithmic Higgs sheaves
    \item Campana-P\u aun's pseudo-effectivity result on the log cotangent bundle
\end{enumerate}
Section \ref{sec: Proof of Logarithmic Base Change Theorem} is devoted to the proof of the logarithmic base change theorem. Section \ref{sec: logarithmic Higgs sheaves} explains the construction of logarithmic Higgs sheaves, slightly modifying that of Popa-Schnell \cite{PS17}. Section \ref{sec: Campana-Paun} explains the results of Campana-P\u aun \cite{CP19}.

The rest of the paper is mainly devoted to the applications of the above three main components. Section \ref{sec: proofs} gives proofs of the theorems on the smooth projective descent of the positivity of the log canonical divisor. Section \ref{sec: additivity of logarithmic Kodaira dimension} explains the conjecture of Popa on the superadditivity of the logarithmic Kodaira dimension and its proof assuming some conjectures of the minimal model program. Section \ref{sec: uniruledness} proves some uniruledness of fibrations over projective spaces. Section \ref{sec: further remarks} discusses some interesting boundary examples regarding the lower bound of the degree of the discriminant locus in Corollary \ref{cor: degree of discriminant locus}.

\subsubsection*{\textnormal{\textbf{Acknowledgements.}}}

I am deeply grateful to my advisor, Mihnea Popa, for his constant support, invaluable advice, and suggestions. I am also indebted to Joe Harris for his support. I thank Jungkai Chen, Yoon-Joo Kim, Fanjun Meng, and Justin Sawon for helpful discussions, and Dori Bejleri for answering my questions and providing references regarding Section \ref{sec: further remarks}. I am supported by the 17th Kwanjeong Study Abroad Scholarship.

\section{Logarithmic base change theorem}
\label{sec: Proof of Logarithmic Base Change Theorem}

The main idea for the logarithmic base change theorem comes from the equality condition in Viehweg's base change theorem.

\subsection{Equality in Viehweg's base change theorem}
Let \(f:X\to Y\) and \(g:Y'\to Y\) be morphisms of smooth quasi-projective varieties. Under the assumptions of Viehweg's base change theorem \ref{thm: base change theorem}, it is easy to show that the equality conditions of the inclusions are obtained when the fiber product \(X'=X\times_YY'\) has canonical singularities. This turns out to be the case when the morphism \(f\) or \(g\) is semistable in a neighborhood of a codimension \(1\) point of \(Y\) after further birational modifications.

We briefly sketch the proof of the equality conditions in Theorem \ref{thm: base change theorem} following Mori \cite{Mori87}*{(4.10)} in this paragraph. Over a neighborhood of a codimension \(1\) point \(y\in Y\), assume \(f\) is semistable, or equivalently the fiber has reduced normal crossing singularities. Let \(D\) be the divisor associated with \(y\). Since a birational modification of the source does not change the pushforward of the pluri-canonical bundle, we replace \(Y'\) with a log resolution of pair \((Y',g^{-1}(D))\). Accordingly, we assume \(g^{-1}(D)\) to be a normal crossing divisor. Then, the fiber product \(X'\) is locally analytically isomorphic to the hypersurface defined by the equation
\[
x_1\cdots x_n=y_1^{f_1}\cdots y_m^{f_m}
\]
for some positive integers \(f_1,\dots, f_m\). This singularity is normal and toric, and thus is canonical, attaining equality of the inclusions in Theorem \ref{thm: base change theorem}.

We will see later that the singularities of the fiber product \(X'\) in the logarithmic setting are locally analytically isomorphic to binomial hypersurface singularities, defined by the equation
\[
x_1^{e_1}\cdots x_n^{e_n}=y_1^{f_1}\cdots y_m^{f_m}
\]
for some positive integers \(e_1,\dots,e_n\) and \(f_1,\dots, f_m\). We analyze these singularities using extra resolution techniques and prove Theorem \ref{thm: logarithmic base change theorem} as a consequence.

\subsection{Fiber product of morphisms of log smooth pairs}

A morphism \(f:(X,E)\to (Y,D)\) of pairs is \textit{strict} if \(E=f^{-1}(D)\). Let \(E=\cup_{i\in \left\{1,\dots,k\right\}} E_i\) be a simple normal crossing divisor of \(X\) with smooth irreducible components \(E_i\). Then the stratum \(E_I\) of \(E\) for a nonempty subset \(I\subset \left\{1,\dots,k\right\}\) is defined by \(E_I=\cap_{i\in I} E_i\).

\begin{defn}
A \textit{strict} morphism of pairs \(f:(X,E)\to (Y,D)\) is \textit{strictly smooth} if \(D\) is a smooth divisor and \(E=f^{-1}(D)\) is a simple normal crossing divisor, such that \(f|_{X\setminus E}:X\setminus E\to Y\setminus D\) is smooth and \(f|_{E_I}:E_I\to D\) is smooth for every stratum \(E_I\) of \(E\).
\end{defn}

In other words, a morphism of pairs \(f:(X,E)\to (Y,D)\) is \textit{strictly smooth} if the morphism is locally analytically (or \'etale locally) equivalent to
\begin{gather*}
    (\C^{a+k},(x_{a+1}\cdots x_{a+k}))\to (\C^{b+1},(t)),\quad a\ge b\\
    (x_1,\dots,x_{a},x_{a+1},\dots,x_{a+k})\mapsto (x_1,\dots,x_b,x_{a+1}^{e_1}\cdots x_{a+k}^{e_k}),
\end{gather*}
where \(x_1,\dots, x_{a+k}\) are local coordinates of \(X\) with divisor \(E=V(x_{a+1}\cdots x_{a+k})\) and \(x_1,\dots,x_b,t\) are local coordinates of \(Y\) with divisor \(D=V(t)\). The local equation for \(E_i\) is \(x_{a+i}\) and \(f^*D=\sum_{i=1}^k e_iE_i\). It is easy to check that a strictly smooth morphism of pairs is a flat morphism.

\begin{lem}
\label{lem: logarithmic base change theorem}
Let \((X,E), (Y,D), (Y',D')\) be quasi-projective log smooth pairs. Let \(f:(X,E)\to (Y,D)\) be a strict morphism of pairs and \(g:(Y',D')\to (Y,D)\) be a strictly smooth morphism of pairs. Consider the commutative diagram~\eqref{log diagram} in Theorem \ref{thm: logarithmic base change theorem}:
\begin{equation*}
\xymatrix{
{(X,E)}\ar[d]_-{f}& {(X',E')}\ar[l]_-{g'}\ar[d]_-{f'}& {(X'',E'')}\ar[l]_-{\mu}\ar[ld]^-{f''}\\
{(Y,D)}&{(Y',D')}\ar[l]^-{g}&\\
}
\end{equation*}
where \(X'=X\times_Y Y'\), \(E'=g'^{-1}(E)\) and \(\mu:(X'',E'')\to (X',E')\) is a log resolution of pairs with \(X''\setminus E''\isom X'\setminus E'\). Then there exists a natural inclusion
\[
\w_{X'}(E+D'-D)^{\tensor N}\subset \mu_*\left(\w_{X''}(E'')^{\tensor N}\right),
\]
where \(\w_{X'}(E+D'-D):=\w_{X'}\tensor g'^*\O_X(E)\tensor f'^*\O_{Y'}(D')\tensor (g\circ f')^*\O_Y(-D)\).
\end{lem}

The proof of Lemma \ref{lem: logarithmic base change theorem} will be given in Section \ref{subsec: singularities of binomial hypersurfaces}. We first give a locally analytic description of \(X'\). Let \(x'\in X'\) be a point. Suppose \(E\) is locally \(x_1\cdots x_n=0\) where \(x_i\)'s are local coordinates at \(g'(x')\in X\), and \(D'\) is locally \(y_1\cdots y_m=0\) where \(y_j\)'s are local coordinates at \(f'(x')\in Y'\).
Assume \(t=0\) is a local equation of \(D\) at \((g\circ f')(x')\) and \(f^*t=x_1^{e_1}\cdots x_n^{e_n}\), \(g^*t=y_1^{f_1}\cdots y_m^{f_m}\). From the above locally analytic description of a strictly smooth morphism, \(g\) is locally analytically isomorphic to
\begin{gather*}
\C^m\times \C^b \times \C^{b'}\to \C\times \C^b,\\
(y_1,\dots,y_m)\times p\times q \mapsto (y_1^{f_1}\cdots y_m^{f_m})\times p.
\end{gather*}
Therefore, \(X'\) is locally analytically isomorphic to the binomial hypersurface \(H\) in \(\C^{n+m+r}\) defined by
\begin{equation}
\label{eqn: binomial hypersurface}
H: x_1^{e_1}\cdots x_n^{e_n}=y_1^{f_1}\cdots y_m^{f_m}
\end{equation}
for some \(r\ge 0\) and positive integers \(e_1,\dots,e_n\), \(f_1,\dots, f_m\). As the local equation suggests, \(X'\) is not normal in general.

Conventionally, singularities of pairs \((X,\Delta)\) are defined when \(X\) is a normal variety and \(\Delta\) is a formal \(\Q\)-linear combination of divisors. However, the description of semi-log canonical pairs in \cite{Kollar13}*{Definition 5.10} and the description of a canonical sheaf for a \(G_1\) and \(S_2\)-variety in \cite{Kovacs13}*{Section 5} suitably generalize the definition of singularities of pairs on a \(G_1\) and \(S_2\)-variety. In particular, we make precise what it means for the pair \((X',E+D'-D)\) to be log canonical, and reduce the proof of Lemma \ref{lem: logarithmic base change theorem} to proving that the pair \((X',E+D'-D)\) is indeed log canonical.

Readers comfortable with singularities of pairs for non-normal Gorenstein varieties may skip the following section.

\subsection{Singularities of pairs: \texorpdfstring{\(S_2\)}{S2}-varieties, Gorenstein over codimension 1}

Let \((X,\Delta)\) be a pair consisting of an \(S_2\)-variety \(X\), Gorenstein over codimension \(1\), and a formal \(\Q\)-linear sum \(\Delta\) of generically nonzero (rational) sections of line bundles on a Zariski open \(j:U\hookrightarrow X\) satisfying \(\codim_X X\setminus U\ge2\) and \(\w_U=\w_X|_U\) invertible:
\[
\Delta=\sum_{i\in I} a_i (\O_U\dashrightarrow \L_i).
\]
Here, \(I\) is a finite set, \(a_i\) is a rational number, \(\L_i\) is a line bundle on \(U\), and a rational section \(\O_U\dashrightarrow \L_i\) is a section defined on a Zariski dense open subset of \(X\). The following remark gives a sheaf theoretic description of \(\Delta\) as a formal \(\Q\)-linear sum of global sections. This contains a formal \(\Q\)-linear sum of Cartier divisors.

\begin{rmk}
Let \(\O_X^*\) be the sheaf of invertible elements in \(\O_X\), \(\mathcal R_X\) be the sheaf of rational sections of \(\O_X\), and \(\mathcal R_X^*\) be the sheaf of generically nonzero rational sections of \(\O_X\). Precisely, for an open subvariety \(U\subset X\), \(\mathcal R_X(U)\) is the product of the rational function fields at the generic points of \(U\) and \(\mathcal R_X^*(U)\) is the subset of elements of \(\mathcal R_X(U)\), nonzero at every generic point of \(U\).

Up to multiplication by a global section of \(\O_U^*\), a generically nonzero rational section \(\O_U\dashrightarrow \L\) is equivalent to the global section of \(\mathcal R_U^*/\O_U^*\). Therefore, \(\Delta\) is a formal \(\Q\)-linear sum of global sections of \(\mathcal R^*/\O^*\) on a Zariski open \(U\subset X\) with \(\codim_X X\setminus U\ge2\). This includes a formal \(\Q\)-linear sum of Cartier divisors on \(U\), or equivalently, a formal \(\Q\)-linear sum of the global sections of \(\mathcal K^*/\O^*\) on \(U\). Here, \(\mathcal K^*_X\) is the sheaf of invertible elements in \(\mathcal K_X\), which is the sheaf of total quotient rings of \(\O_X\). See Hartshorne \cite{Hartshorne77}*{p.141} for the description of Cartier divisors. In particular, \(\mathcal K^*_X/\O_X^*\subset \mathcal R_X^*/\O_X^*\).
\end{rmk}

Unifying the formal sum by a common denominator of \(a_i\)'s, the sum is alternatively expressed as a single term \(\Delta=a(\O_U\dashrightarrow \L)\), for some rational number \(a\) and some rational section of a line bundle \(\L=\bigotimes_{i\in I}\L_i^{\tensor Na_i}\) with sufficiently divisible \(N\). In particular, \((\O_U\dashrightarrow \L_i)\) is equivalent to \(-(\O_U\dashrightarrow \L_i^{-1})\), where the latter rational section is the reciprocal of the former rational section.

Assume that \(\w_U^{\tensor N}(N\Delta):=\w_U^{\tensor N}\tensor\bigotimes_{i\in I}\L_i^{\tensor Na_i}\) extends to a line bundle on \(X\) for some \(N>0\). Since \(X\) is an \(S_2\)-variety, this line bundle is uniquely determined by
\[
\w_X^{\tensor N}(N\Delta):= j_*\w_U^{\tensor N}(N\Delta).
\]
Let \(\mu:\widetilde X\to X\) be a resolution of singularities of \(X\). Over a Zariski dense open subset of \(X\), we have a rational morphism of line bundles \(\mu^*\w_X^{\tensor N}\dashrightarrow \w_{\widetilde X}^{\tensor N}\), and pullbacks of generically nonzero rational sections \((\O_U\dashrightarrow \L_i)\). Therefore, there exists a canonically defined rational morphism of line bundles
\[
\mu^*\w_X^{\tensor N}(N\Delta)\dashrightarrow \w_{\widetilde X}^{\tensor N},
\]
which is an isomorphism on a Zariski dense open subset of \(\widetilde X\). As a result, we obtain a generically nonzero rational section \(s:\O_{\widetilde X}\dashrightarrow\w_{\widetilde X}^{\tensor N}\tensor [\mu^*\w_X^{\tensor N}(N\Delta)]^{-1}\) on \(\widetilde X\).

Let \(E\) be a divisor on \(\widetilde X\). The discrepancy of \(E\) with respect to the pair \((X,\Delta)\) is defined as
\[
a(E,X,\Delta):=\frac{\textrm{ord}_E(s)}{N}.
\]
It is easy to verify that \(a(E,X,\Delta)\) is well defined, independent of the choice of \(N\) or a resolution of singularities. In conclusion, we define singularities of pairs (i.e. log canonical, log terminal, e.t.c.) as in the normal case.

Under the setting of Lemma \ref{lem: logarithmic base change theorem}, it suffices to prove that the pair \((X',E+D'-D)\) is log canonical. More specifically, the pair \((X',E+D'-D)\) consists of a Gorenstein variety \(X'\) and a formal sum of sections of line bundles
\[
(\O_{X'}\to g'^*\O_X(E))+(\O_{X'}\to f'^*\O_{Y'}(D'))-(\O_{X'}\to (g\circ f')^*\O_Y(D)),
\]
where the section \((\O_{X'}\to g'^*\O_X(E))\) is induced by the pullback of the section \(\O_X\to \O_X(E)\). Therefore, \((X',E+D'-D)\) is log canonical if and only if there exists a natural inclusion
\[
\mu^*\left(\w_{X'}(E+D'-D)^{\tensor N}\right)\subset \w_{X''}(E'')^{\tensor N},
\]
which is equivalent to the conclusion of Lemma \ref{lem: logarithmic base change theorem} due to the adjointness of the pair \((\mu^*,\mu_*)\).

\subsection{Singularities of binomial hypersurfaces: Proof of Lemma \ref{lem: logarithmic base change theorem}}
\label{subsec: singularities of binomial hypersurfaces}

We prove \((X',E+D'-D)\) is log canonical, which implies Lemma \ref{lem: logarithmic base change theorem}. From equation \eqref{eqn: binomial hypersurface}, the pair \((X',E+D'-D)\) is locally analytically equivalent to
\[
\left(H: x_1^{e_1}\cdots x_n^{e_n}=y_1^{f_1}\cdots y_m^{f_m},\quad  D(x_1\cdots x_n)+D(y_1\cdots y_m)-D(t)\right),
\]
where \(D\) is to indicate sections of line bundles defined by multiplications of \(x_1\cdots x_n\), \(y_1\cdots y_m\), or \(t=x_1^{e_1}\cdots x_n^{e_n}=y_1^{f_1}\cdots y_m^{f_m}\). Notice that the singularities of pairs can be checked locally analytically (or \'etale locally), so it suffices to prove \[\left(H, D(x_1\cdots x_n)+D(y_1\cdots y_m)-D(t)\right)\] is log canonical, which is implied by the following:

\begin{prop}
\label{prop: log canonical}
For a binomial hypersurface \(H:\left\{t=x_1^{e_1}\cdots x_n^{e_n}=y_1^{f_1}\cdots y_m^{f_m}\right\}\subset\A^{n+m+r}\) with coordinates \(x_1,\dots,x_n, y_1,\dots,y_m,z_1,\dots z_r\), the pair
\[
\left(H: x_1^{e_1}\cdots x_n^{e_n}=y_1^{f_1}\cdots y_m^{f_m},\quad D(x_1\cdots x_n y_1\cdots y_mz_1\cdots z_r)-D(t)\right)
\]
is log canonical.
\end{prop}

Let \(w=x_1\cdots x_n y_1\cdots y_mz_1\cdots z_r\) be the multiplication of all the coordinates. The reason for adding \(D(z_1\cdots z_r)\) into the pair is to apply induction. In the remaining section, we resolve the singularity of \(H\) through a sequence of blow-ups along the ideals cut out by coordinate sections, such as \((x_1,\dots, x_k,y_1\dots, y_l)\). Denote \(Z:=V(x_1,\dots, x_k,y_1\dots, y_l)\). Recall that
\[
Bl_{Z}\A^{n+m+r}=\textrm{Proj } \frac{k[x_1,\dots, x_n, y_1,\dots, y_m, z_1\dots z_r][X_1,\dots, X_k, Y_1,\dots, Y_l]}{\left(\frac{X_1}{x_1}=\dots=\frac{X_k}{x_k}=\frac{Y_1}{y_1}=\dots=\frac{Y_l}{y_l}\right)}
\]
with the induced morphism \(\beta:Bl_{(x_1,\dots, x_k,y_1\dots, y_l)}\A^{n+m+r}\to \A^{n+m+r}\).
Restricted to the affine chart \(\A^{n+m+r}:X_1\neq 0\), we obtain the transformation
\[
\beta:\A^{n+m+r}\to \A^{n+m+r}
\]
given by
\begin{equation}
\label{eqn: blow up transformation}
\begin{aligned}
x_1\mapsto x_1, x_2\mapsto x_1x_2,\dots, x_k\mapsto x_1x_k,\\
y_1\mapsto x_1y_1,y_2\mapsto x_1y_2,\dots, y_l\mapsto x_1y_l,
\end{aligned}
\end{equation}
and all the other coordinates remain the same. Therefore, the pullback \(\beta^*H\) of \(H\) in the chart \(X_1\neq 0\) is the binomial hypersurface
\[
x_1^{e_1+\dots+e_k}x_2^{e_2}\cdots x_n^{e_n}=x_1^{f_1+\dots+f_l}y_1^{f_1}\cdots y_m^{f_m}.
\]
The strict transform \(\widetilde H\) of \(H\) in the chart \(X_1\neq 0\) is the binomial hypersurface above divided by \(x_1^{\min(e_1+\dots+e_k, f_1+\dots+f_l)}\).

Therefore, for each affine chart, we may consider a pair \((\widetilde H,D(\tilde w)-D(\tilde t))\) where \(\tilde w\) is the multiplication of the coordinates and \(\tilde t\) is the monomial appearing in the binomial equation.

\begin{lem}
\label{lem: blow up binomial}
In the above notation, let \(\beta:\widetilde H\to H\). On the affine chart \(X_1\neq 0\),
\begin{enumerate}
    \item If \(e_1+\dots+e_k \ge f_1+\dots+f_l\), then \(\widetilde H\) is a binomial hypersurface
    \[
    x_1^{(e_1+\dots+e_k)-(f_1+\dots+f_l)}x_2^{e_2}\cdots x_n^{e_n}=y_1^{f_1}\cdots y_m^{f_m}.
    \]
    \item If \(e_1+\dots+e_k \le f_1+\dots+f_l\), then \(\widetilde{H}\) is a binomial hypersurface
    \[
    x_2^{e_2}\cdots x_n^{e_n}=x_1^{-(e_1+\dots+e_k)+(f_1+\dots+f_l)}y_1^{f_1}\cdots y_m^{f_m}.
    \]
\end{enumerate}
In either of the cases, there exists a canonical isomorphism
\[
\beta^*[\w_H(D(w)-D(t))]\simeq \w_{\widetilde H}(D(\tilde w)-D(\tilde t))
\]
In particular, \((H,D(w)-D(t))\) is log canonical if and only if \((\widetilde H,D(\tilde w)-D(\tilde t))\) is log canonical on every affine chart.
\end{lem}

More specifically, the canonical morphism indicates that
\[
a(E,\widetilde H,D(\tilde w)-D(\tilde t))= a(E,H,D(w)-D(t)),
\]
for all divisors \(E\) appearing in a resolution of \(\widetilde H\).

\begin{proof}
We denote \(Bl\A:=Bl_{(x_1,\dots, x_k,y_1\dots, y_l)}\A^{n+m+r}\) and \(\A:=\A^{n+m+r}\) with the blow-up morphism \(\beta:Bl\A\to \A\), in short.
On the chart \(X_1\neq 0\), the exceptional divisor is given by \(x_1=0\). Therefore,
\[
\w_{Bl\A}=\beta^*\w_{\A}+(k+l-1)D(x_1).
\]
By the adjunction formula, there exist natural isomorphisms \(\w_{\widetilde H}=\w_{Bl\A}(\widetilde H)|_{\widetilde H}\) and \(\w_H=\w_{\A}(H)|_{H}\). Since
\[
\beta^*(H)=(\widetilde H) + \min(e_1+\dots+e_k, f_1+\dots+f_l)D(x_1),
\]
we have a natural isomorphism
\[
\w_{\widetilde H}=\beta^*\w_H+\left(k+l-1-\min(e_1+\dots+e_k, f_1+\dots+f_l)\right)D(x_1).
\]
From the transformation relation \eqref{eqn: blow up transformation}, we have
\[\beta^*D(x_1\cdots x_ny_1\cdots y_m z_1\cdots z_r)=D(x_1\cdots x_ny_1\cdots y_m z_1\cdots z_r)+(k+l-1)D(x_1)\]
and in either case (1) or (2), it is easy to check that
\[\beta^*D(t)=D(\tilde t)+\min(e_1+\dots+e_k, f_1+\dots+f_l)D(x_1).\]
Combining all of the above, we obtain a canonical isomorphism
\[
\beta^*[\w_H(D(w)-D(t))]\simeq \w_{\widetilde H}(D(\tilde w)-D(\tilde t)),
\]
which completes the proof.
\end{proof}

By symmetry, Lemma \ref{lem: blow up binomial} holds for every affine chart. We finish the proof of Proposition \ref{prop: log canonical} using Bierstone-Milman's algorithm \cite{BM06}*{Section 5} for resolution of singularities of binomial varieties along with Lemma \ref{lem: blow up binomial}. Define 
\begin{gather*}
m(H):=\min(e_1+\dots+e_n,f_1+\dots +f_m),\\M(H):=\max(e_1+\dots+e_n,f_1+\dots +f_m).
\end{gather*}
Without loss of generality, assume \(e_1+\dots+e_n\ge f_1+\dots +f_m\). Choose a subset \(S\subset \left\{1,\dots, n\right\}\) such that 
\[
0\le \sum_{s\in S}e_s-m(H)<\min_{s\in S}\left\{e_s\right\}.
\]
For convenience, assume \(S=\left\{1,\dots, k\right\}\) for some \(1\le k\le n\). Let a subvariety \(Z:x_1=\dots=x_k=y_1=\dots=y_m=0\), and we now perform an algorithm of blowing up along \(Z\): \(Bl_Z \A^{n+m+r}\to \A^{n+m+r}\). The strict transform \(\widetilde H\) of \(H\) of the blow-up is the binomial hypersurface
\begin{gather*}
x_1^{e_1+\dots+e_k-m(H)}x_2^{e_2}\cdots x_n^{e_n}=y_1^{f_1}\cdots y_m^{f_m} \quad \textrm{on the chart } X_1\neq0,\\
y_1^{e_1+\dots+e_k-m(H)}x_1^{e_1}\cdots x_n^{e_n}=y_2^{f_2}\cdots y_m^{f_m} \quad \textrm{on the chart } Y_1\neq0.
\end{gather*}
On every affine chart of the blow-up, the lexicographic ordering of the pair \[(m(H), M(H))\] decreases. Therefore, the algorithm terminates on each affine chart, in which case, we obtain a hypersurface of the form
\[
H:\left\{x_1^{e_1}\cdots x_n^{e_n}=1\right\} \subset \A^{n+m+r}.
\]
It is now obvious that the pair \((H,D(w)-D(t))\) is log canonical. Indeed, \(H\) is smooth, \(D(t)\) is a zero divisor, and \(D(w)\) is a simple normal crossing divisor defined by the product of the coordinates except for \(x_i\)'s. Therefore, by Lemma \ref{lem: blow up binomial}, the original \((H, D(w)-D(t))\) is log canonical by induction on \((m(H),M(H))\subset \Z^{\ge0}\times \Z^{\ge0}\) endowed with the lexicographic ordering. This completes the proof of Proposition \ref{prop: log canonical}, and thus Lemma \ref{lem: logarithmic base change theorem}.

\subsection{Proof of the logarithmic base change theorem}

Here is an immediate consequence of Lemma \ref{lem: logarithmic base change theorem}.

\begin{cor}
\label{cor: logarithmic base change theorem}
In the setting of Lemma \ref{lem: logarithmic base change theorem}, we have the following inclusion:
\[
g^*f_*\left(\w_X(E)/\w_Y(D)^{\tensor N}\right)\subset f''_*\left(\w_{X''}(E'')/\w_{Y'}(D')^{\tensor N}\right).
\]
\end{cor}

\begin{proof}
Since \(g\) is flat, \(g'^*\w_{X/Y}=\w_{X'/Y'}\) which implies that
\begin{align*}
g'^*\left(\w_X(E)/\w_Y(D)^{\tensor N}\right)&=\left(\w_{X'}(E-D)/\w_{Y'}\right)^{\tensor N}\\
&=\left(\w_{X'}(E+D'-D)/\w_{Y'}(D')\right)^{\tensor N}\\
&\subset \mu_*\left(\w_{X''}(E'')/\w_{Y'}(D')^{\tensor N}\right).
\end{align*}
Taking pushforward \(f'_*\) on each side, we obtain the inclusion, since \(g^*f_*=f'_*g'^*\) by the flatness of \(g\).
\end{proof}

\begin{proof}[Proof of Theorem \ref{thm: logarithmic base change theorem}]
As given in the assumption, \(g:(Y',D')\to (Y,D)\) is a strict morphism of log smooth pairs such that \(g|_{Y'\setminus D'}:Y'\setminus D' \to Y\setminus D\) is smooth. Then the set of points on \(D\), such that either \(D\) is singular or \(g|_{D_I'}:D_I'\to D\) is singular for some stratum \(D_I'\) of \(D'\), is a Zariski closed subset of \(Y\) with codimension at least \(2\). Therefore, there exists an open subvariety \(Y_0\subset Y\) with \(\codim_{Y}Y\setminus Y_0\ge 2\), such that \(g|_{g^{-1}(Y_0)}:(Y'_0,D'_0)\to (Y_0,D_0)\) is strictly smooth. Since we are taking the double dual at the end, we need only prove the inclusion in the statement of the theorem on \(Y_0\). Therefore, it suffices to prove the theorem when \(g\) is a strictly smooth morphism of log smooth pairs.

Then by Corollary \ref{cor: logarithmic base change theorem}, we have 
\[
g_*\left[g^*f_*\left(\w_X(E)/\w_Y(D)^{\tensor N}\right)\tensor \left(\w_{Y'}(D')/\w_Y(D)^{\tensor N}\right)\right]\subset h_*\left(\w_{X''}(E'')/\w_{Y}(D)^{\tensor N}\right).
\]
Further shrinking an open subvariety \(Y_0\subset Y\) with \(\codim_Y Y\setminus Y_0\ge 2\) where
\[f_*\left(\w_X(E)/\w_Y(D)^{\tensor N}\right)\]
is locally free, we have the following inclusion
\[
f_*\left(\w_X(E)/\w_Y(D)^{\tensor N}\right)\tensor g_* \left(\w_{Y'}(D')/\w_Y(D)^{\tensor N}\right)\subset h_*\left(\w_{X''}(E'')/\w_{Y}(D)^{\tensor N}\right)
\]
from which we obtain the conclusion.
\end{proof}

\section{Logarithmic Higgs sheaves and Campana-P\u aun criterion}
\label{sec: logarithmic Higgs sheaves and Campana Paun}

In the next two subsections, we mostly review some important constructions from the literature, with small modifications needed here in the first one. In the end, we prove Theorem \ref{thm: effectivity descends to pseudo-effectivity} (2) as an immediate application.

\subsection{Construction of logarithmic Higgs sheaves}
\label{sec: logarithmic Higgs sheaves}

Given a morphism from a smooth projective variety to a curve, Viehweg-Zuo \cite{VZ01} studied a lower bound of the number of singular fibers via the construction of Higgs bundles. This construction was later generalized in \cite{VZ02} to a morphism of higher dimensional varieties with positivity conditions on the fibers. Using the theory of Hodge modules, Popa-Schnell \cite{PS17} refined the construction of logarithmic Higgs sheaves with poles along the discriminant divisor in general. Along with the results of Campana-P\u{a}un \cite{CP19}, these provide a powerful method to study the hyperbolicity of the base of a smooth projective morphism. In this section, we summarize and explain the above, following Popa and Schnell's treatment with some simplifications. As before, \((Y,D)\) is a log smooth pair.

\begin{defn}
A graded \(\O_Y\)-module \(\F_\bullet=\oplus_k \F_k\) is a \textit{graded logarithmic Higgs sheaf} with poles along \(D\) if there exists a logarithmic Higgs structure 
\[
\phi_\bullet:\F_\bullet\to \F_{\bullet+1}\tensor \Omega_Y(\log D)
\]
such that \(\phi\wedge \phi:\F_{\bullet}\to \F_{\bullet+2}\tensor \Omega_Y^2(\log D)\) is the zero morphism. Unless otherwise stated, \(\F_k=0\) for \(k\ll 0\). Define
\[
\K_k(\phi):=\ker (\phi_k:\F_k\to \F_{k+1}\tensor \Omega_Y(\log D))
\]
to be the kernel of the generalized Kodaira-Spencer map for each \(k\).
\end{defn}

When a surjective morphism \(f:X\to Y\) with connected fibers satisfies some additional effectivity condition for a particular line bundle, Popa-Schnell \cite{PS17}*{Theorem 2.3} constructed a graded logarithmic Higgs sheaf with poles along the discriminant locus. For our use later in the paper, we modify \cite{PS17}*{Theorem 2.3} for a logarithmic setting, and drop the assumption on the connectedness of the fibers.

\begin{thm}
\label{thm: construction of logarithmic Higgs sheaves}
Let \(f:(X,E)\to (Y,D)\) be a surjective projective morphism of quasi-projective log smooth pairs with \(E=f^{-1}(D)\), and \(f|_{X\setminus E}\) is smooth. For a line bundle \(\L\) on \(Y\), assume that some power of \(\w_{X}(E)/\w_Y(D)\tensor f^*\L^{-1}\) is effective (i.e. has a nonzero section). Then there exists a graded logarithmic Higgs sheaf \(\F_\bullet\) with poles along \(D\) satisfying the following properties:
\begin{enumerate}[(a)]
\item One has \(\L\subset \F_0\), and \(\F_k=0\) for \(k<0\).
\item There exists \(d\) such that \(\F_k=0\) for all \(k>d\).
\item Each \(\F_k\) is a reflexive coherent sheaf on \(Y\).
\item Each dual \(\K_k(\phi)^{*}\) of the kernel of the generalized Kodaira-Spencer map is weakly positive if \(\K_k(\phi)\neq 0\).
\end{enumerate}
\end{thm}

We recall the definitions and properties of weakly positive sheaves and big sheaves, introduced by Viehweg \cites{Viehweg83, Viehweg83II}; cf. also \cite{Mori87}*{Section 5}. In what follows, \(\hat{S}^\alpha F\) denotes the double dual of the \(\alpha\)-symmetric product of \(F\), and \(\widehat\det(F)\) denotes the double dual of the determinant of \(F\).

\begin{defn}
\label{defn: weak positivity and bigness}
A torsion free coherent sheaf \(F\) on a normal quasi-projective variety \(W\) is \textit{weakly positive} if, for every positive integer \(\alpha\) and every ample line bundle \(H\), there exists a positive integer \(\beta\) such that \(\hat{S}^{\alpha\beta}F\tensor H^\beta\) is generically generated by global sections. We say \(F\) is \textit{big} if, for every ample line bundle \(H\), there exists a positive integer \(\alpha\) such that \(\hat{S}^{\alpha}F\tensor H^{-1}\) is weakly positive.
\end{defn}

If \(F\) is a line bundle and \(W\) is a normal projective variety, then \(F\) is weakly positive (resp. big) if and only if \(F\) is pseudo-effective (resp. big). This is immediate from the definition. We recall additional important properties of weakly positive and big sheaves.

\begin{lem}[\cites{Viehweg83, Viehweg83II, Kawamata85, Mori87}]
\label{lem: weakly positive and big sheaves}
Let \(F\) be a nonzero torsion free coherent sheaf on a normal quasi-projective variety \(W\). Then we have the following properties.
\begin{enumerate}[(1)]
\item Let \(W_0\subset W\) be a Zariski open subset. If \(F\) is weakly positive (resp. big), then \(F|_{W_0}\) is weakly positive (resp. big). If \(\codim_{W} W\setminus W_0\ge 2\), then the converse is true.
\item Let \(F\to G\) be a generically surjective morphism of nonzero torsion free coherent sheaves. If \(F\) is weakly positive (resp. big), then \(G\) is weakly positive (resp. big).
\item Let \(A\) be a big line bundle. If \(F\) is weakly positive, then \(F\tensor A\) is big.
\item If \(F\) is weakly positive (resp. big), then \(\widehat\det(F)\) is weakly positive (resp. big).
\item If \(\tau:W'\to W\) is a finite surjective morphism of normal varieties, then \(F\) is weakly positive (resp. big) if and only if \(\tau^*F\) is weakly positive (resp. big).
\end{enumerate}
\end{lem}

We include a sketch of the proof of Theorem \ref{thm: construction of logarithmic Higgs sheaves} for completeness.

\begin{proof}[Proof of Theorem \ref{thm: construction of logarithmic Higgs sheaves}]
Notice that it suffices to construct \(\F_\bullet\) on the complement of a codimension \(2\) subvariety in \(Y\), satisfying the properties \((a)\)-\((d)\). Indeed, taking the reflexive hull of \(\F_\bullet\) over \(Y\), we obtain a graded logarithmic Higgs sheaf, still satisfying the properties of Theorem \ref{thm: construction of logarithmic Higgs sheaves} (see Lemma \ref{lem: weakly positive and big sheaves} (1) for the property \((d)\)).

Take a resolution \(\psi:Z\to X\) of the cyclic cover of \(X\) induced by the section of some power of \(B:=\w_{X}(E)/\w_Y(D)\tensor f^*\L^{-1}\). In particular, there exists a natural inclusion \(\psi^*B^{-1}\to \O_Z\). Denote \(h:=f\circ\psi:Z\to Y\), and let \(D_h\) be the union of the discriminant locus \(\Delta(h)\) and \(D\) in \(Y\). After removing a codimension \(2\) subvariety in \(Y\) and taking a suitable resolution \(Z\) of the cyclic cover, we may assume that \(D_h\) is a smooth divisor such that \(E_h:=h^{-1}D_h\) is a simple normal crossing divisor of \(Z\).

Let \(\Omega_{X/Y}(\log E)\) (resp. \(\Omega_{Z/Y}(\log E_h)\)) be the cokernel of the natural inclusion of the logarithmic cotangent bundles
\[
f^*\Omega_{Y}(\log D)\to \Omega_{X}(\log E) \quad (\mathrm{resp. } \, h^*\Omega_{Y}(\log D_h)\to \Omega_{Z}(\log E_h)).
\]
Further removing a codimension \(2\) subvariety in \(Y\) so that \(f:(X,E)\to (Y,D)\) and \(h:(Z,E_h)\to (Y,D_h)\) are strictly smooth (see e.g. Proof of Theorem \ref{thm: logarithmic base change theorem}), we may assume that \(\Omega_{X/Y}(\log E)\) and \(\Omega_{Z/Y}(\log E_h)\) are locally free from the locally analytic description. As a consequence, we have the Koszul filtration
\[
\Koz^q \, \Omega_{X}^i(\log E):= \mathrm{image}\, \left[f^*\Omega_{Y}^{q}(\log D)\tensor \Omega_{X}^{i-q}(\log E)\to \Omega_{X}^i(\log E)\right]
\]
where \(\Omega_{X}^i(\log E):=\wedge^i\Omega_X(\log E)\) is the \(i\)-th exterior power, for all \(0\le i\le \dim X\). Hence, we have the natural isomorphism
\[\Koz^q/\Koz^{q+1}\, \Omega_{X}^i(\log E)\isom f^*\Omega_{Y}^{q}(\log D)\tensor \Omega_{X/Y}^{i-q}(\log E)\]
and the tautological short exact sequence
\[
0\to f^*\Omega_{Y}(\log D)\tensor \Omega_{X/Y}^{i-1}(\log E) \to \Koz^0/\Koz^{2}\, \Omega_{X}^i(\log E) \to \Omega_{X/Y}^{i}(\log E) \to 0,
\]
which we denote by \(\mathcal C^i_{X/Y}(\log E)\). Respectively, we have the tautological short exact sequence \(\mathcal C^i_{Z/Y}(\log E_h)\).

From the following natural morphism of complexes,
\begin{displaymath}
\xymatrix{
{0} \ar[r] & {h^*\Omega_{Y}(\log D)} \ar[r] \ar[d] & {\psi^*\Omega_{X}(\log E)} \ar[r] \ar[d]  & {\psi^*\Omega_{X/Y}(\log E)} \ar[r] \ar[d] & {0} \\
{0} \ar[r] & {h^*\Omega_{Y}(\log D_h)} \ar[r] & {\Omega_{Z}(\log E_h)} \ar[r] & {\Omega_{Z/Y}(\log E_h)} \ar[r] & 0
}
\end{displaymath}
we have the morphism \(\psi^*\mathcal C^i_{X/Y}(\log E)\to \mathcal C^i_{Z/Y}(\log E_h)\) of the tautological short exact sequences, for all \(i\). Tensored with the natural injection \(\psi^*B^{-1}\to \O_Z\), we obtain the following morphism of short exact sequences:
\begin{equation}
\label{eqn: cyclic cover tautological morphism}
\psi^*\left(\mathcal C^i_{X/Y}(\log E)\tensor B^{-1}\right)\to \mathcal C^i_{Z/Y}(\log E_h).  
\end{equation}
Let \(d=\dim X-\dim Y\). From the Leray spectral sequence associated with \(Rh_*\isom Rf_*\circ R\psi_*\) and the adjointness of the pair \((\psi^*,\psi_*)\), we have the natural morphism
\[
R^{d-i}f_*\left(\Omega^i_{X/Y}(\log E)\tensor B^{-1}\right)\to{R^{d-i}h_*\left(\psi^*\left(\Omega^i_{X/Y}(\log E)\tensor B^{-1}\right)\right)},
\]
which induces the following commutative diagram of the connecting homomorphisms via the derived pushforward of \eqref{eqn: cyclic cover tautological morphism}:
\begin{displaymath}
\xymatrix{
{R^{d-i}f_*\left(\Omega^i_{X/Y}(\log E)\tensor B^{-1}\right)} \ar[r] \ar[d]_{\rho_{d-i}}  & {R^{d-i+1}f_*\left(\Omega^{i-1}_{X/Y}(\log E)\tensor B^{-1}\right)\tensor \Omega_Y(\log D)} \ar[d]^{\rho_{d-i+1}\tensor \iota} \\
{R^{d-i}h_*\left(\Omega^i_{Z/Y}(\log E_h)\right)} \ar[r]^-{\phi'_{d-i}} & {R^{d-i+1}h_*\left(\Omega^{i-1}_{Z/Y}(\log E_h)\right)\tensor \Omega_Y(\log D_h)}
}
\end{displaymath}

Due to Steenbrink (see for example \cites{Steenbrink76, Zucker84}), it is well known that \(\phi'_\bullet\) is the logarithmic Higgs structure of the graded logarithmic Higgs bundle
\[
\gr^F_{\bullet}\mathcal V_0\isom\bigoplus_{k=0}^d R^{k}h_*\left(\Omega^{d-k}_{Z/Y}(\log E_h)\right),
\]
where \(\mathcal V_0\) is Deligne's canonical extension along \(D_h\) on \(Y\), with the induced filtration \(F\), associated to the polarized variation of Hodge structure given by the middle cohomologies of the smooth fibers of \(h\). Define
\[
\mathcal F_k:= \left(\mathrm{image}\, \left[\rho_k:{R^{k}f_*\left(\Omega^{d-k}_{X/Y}(\log E)\tensor B^{-1}\right)}\to {R^{k}h_*\left(\Omega^{d-k}_{Z/Y}(\log E_h)\right)} \right]\right)^{**}
\]
for \(0\le k\le d\), and \(\F_k:=0\) otherwise. Then, \(\F_\bullet\) is a graded logarithmic Higgs sheaf with poles along \(D\). Its logarithmic Higgs structure \(\phi_\bullet\) is induced by \(\phi'_\bullet\) of \(\gr^F_{\bullet}\mathcal V_0\). Notice that \(\rho_0\) is the pushforward \(f_*\) of the inclusion
\[
\w_X(E)/\w_Y(D)\tensor B^{-1}\to \psi_*(\w_Z(E_h)/\w_Y(D_h)),
\]
which is the adjoint of \(\psi^*(\w_X(E)/\w_Y(D)\tensor B^{-1})\to \w_Z(E_h)/\w_Y(D_h)\). Accordingly, \(\F_0=(\L\tensor f_*\O_X)^{**}\), which implies \(\L\subset \F_0\). Therefore, the properties \((a)\)-\((c)\) are immediate from the construction.

Observe that \(\F_\bullet\subset \gr^F_\bullet \mathcal V_0\), which implies \(\K_k(\phi)\subset \K_k(\phi')\). Hence, we have a generically surjective morphism \(\K_k(\phi')^*\to \K_k(\phi)^*\). It is well known that \(\K_k(\phi')^*\) is weakly positive if not zero \cites{PW16, Brunebarbe18, Zuo00}. From Lemma \ref{lem: weakly positive and big sheaves} (2), \(\K_k(\phi)^*\) is weakly positive if not zero, which verifies the property (d).
\end{proof}

\begin{rmk}
\label{rmk: assumption in logarithmic Higgs sheaves}
The conclusion of Theorem \ref{thm: construction of logarithmic Higgs sheaves} continues to hold when we replace the assumption \(\k(X, \w_{X}(E)/\w_Y(D)\tensor f^*\L^{-1})\ge 0\) by the existence of a nonzero morphism
\[
\L^{\tensor N}\to \left[f_*\left(\w_{X}(E)/\w_Y(D)^{\tensor N}\right)\right]^{**}
\]
for some positive integer \(N\). By the left-right adjointness of \((f^*, f_*)\), this morphism implies the existence of a nonzero section of \(N\)-th power of \[\w_{X}(E)/\w_Y(D)\tensor f^*\L^{-1}\] over the complement of a codimension 2 subvariety in \(Y\). Therefore, we obtain a graded logarithmic Higgs sheaf \(\F_\bullet\), on the complement of a codimension \(2\) subvariety in \(Y\), satisfying the properties of Theorem \ref{thm: construction of logarithmic Higgs sheaves}. As at the beginning of the proof of Theorem \ref{thm: construction of logarithmic Higgs sheaves}, we obtain the desired graded logarithmic Higgs sheaf via taking the reflexive hull of \(\F_\bullet\) over \(Y\). This refined assumption turns out to be more convenient for the applications of the logarithmic base change theorem.
\end{rmk}

The following lemma is obtained from the standard manipulation of logarithmic Higgs sheaves by Viehweg and Zuo. This allows us to apply Campana and P\u aun's results explained in the next subsection to deduce results on hyperbolicity.

\begin{lem}
\label{lem: CP setup}
Let \((Y,D)\) be a log smooth pair. Let \(\F_\bullet\) be a graded logarithmic Higgs sheaf with poles along \(D\) satisfying the properties (a)-(d) of Theorem \ref{thm: construction of logarithmic Higgs sheaves}. Then there exists a pseudo-effective line bundle \(P\) and a nonzero morphism
\[
\L^{\tensor r}\tensor P\to (\Omega_Y(\log D))^{\tensor kr}
\]
for some \(r>0\), \(k\ge 0\).
\end{lem}

\begin{proof}
From the logarithmic Higgs structure \(\phi_\bullet\) of \(\F_\bullet\), we have a sequence of morphisms
\[
\phi_{k}\tensor id:\F_{k}\tensor (\Omega_Y(\log D))^{\tensor k}\to \F_{k+1}\tensor (\Omega_Y(\log D))^{\tensor k+1}.
\]
Notice that \(\F_{k}=0\) for large enough \(k\). Therefore, the line bundle \(\L\) is contained in the kernel of \(\phi_{k}\tensor id\) for some \(k\ge 0\):
\[
\L\subset \K_{k}(\phi)\tensor (\Omega_Y(\log D))^{\tensor k}.
\]
This implies the existence of a nonzero morphism 
\[
\K_{k}(\phi)^*\to (\Omega_Y(\log D))^{\tensor k}\tensor\L^{-1}.
\]
Let \(\K^*\) be the image of the morphism and \(r\) be the generic rank of \(\K^*\). From the split surjection \({\K^*}^{\tensor r}\to \det \K^*\), where \(\K^*\) is a vector bundle outside of a codimension \(2\) locus in \(Y\), we obtain a nonzero morphism
\[
\widehat\det\K^*\to \left[(\Omega_Y(\log D))^{\tensor k}\tensor \L^{-1}\right]^{\tensor r}.
\]
Note that \(P:=\widehat\det\K^*\) is a pseudo-effective line bundle by Lemma \ref{lem: weakly positive and big sheaves} (4). Therefore we obtain the conclusion.
\end{proof}

\subsection{Positivity properties of the logarithmic cotangent bundle}
\label{sec: Campana-Paun}

We highlight the results of Campana-P\u aun \cite{CP19}. Investigating foliations on orbifold tangent bundles, they studied the positivity properties of their tensor powers. Together with the logarithmic base change theorem and the construction of logarithmic Higgs sheaves, this provides the machinery to study the positivity of the base of a smooth projective morphism.

\begin{thm}[\cite{CP19}*{Theorem 1.3}]
\label{thm: CP main}
Let \((Y,D)\) be a projective log smooth pair such that \(K_Y+D\) is pseudo-effective. For every quotient \(Q\) of a tensor power of the logarithmic cotangent bundle \((\Omega_Y(\log D))^{\tensor N}\) with \(N\ge 1\), the first Chern class \(c_1(Q)\) is pseudo-effective.
\end{thm}

\begin{thm}[\cite{CP19}*{Theorem 7.6}]
\label{thm: CP pseudoeffective}
Let \((Y,D)\) be a projective log smooth pair, and let \(L\) be a pseudo-effective line bundle on \(Y\). If there exists a nonzero morphism
\[
L\to (\Omega_Y(\log D))^{\tensor k}\tensor (\w_Y(D))^{\tensor r}
\]
for some integers \(k\ge 0\) and \(r\ge 1\), then \(K_Y+D\) is pseudo-effective.
\end{thm}

Originally, Campana and P\u aun stated their results in terms of an orbifold pair \((Y,D)\), but we restrict the statements to a log smooth pair since it suffices for our purpose.

As a quick application of the results in this section, we prove Theorem \ref{thm: effectivity descends to pseudo-effectivity} (2).

\begin{proof}[Proof of Theorem \ref{thm: effectivity descends to pseudo-effectivity} (2)]
Since \(\w_X(E)/\w_Y(D)\tensor f^*\w_Y(D)\) is effective, there exists a graded logarithmic Higgs sheaf \(\F_\bullet\) satisfying the properties of Theorem \ref{thm: construction of logarithmic Higgs sheaves}, with \(\L=\w_Y(D)^{-1}\). From Lemma \ref{lem: CP setup}, we have a pseudo-effective line bundle \(P\) and a nonzero morphism
\[
(\w_Y(D))^{\tensor -r}\tensor P\to (\Omega_Y(\log D))^{\tensor kr}
\]
for \(r>0\), \(k\ge 0\). Therefore, \(K_Y+D\) is pseudo-effective by Theorem \ref{thm: CP pseudoeffective}.
\end{proof}

Suppose \(f:X\to \P^n\) is a surjective morphism, with \(X\) a smooth projective variety of non-negative Kodaira dimension. By Theorem \ref{thm: effectivity descends to pseudo-effectivity} (2), it is easy to see that \(\dim \Delta(f)=n-1\) and \(\deg \Delta(f)\ge n+1\). However, this inequality is not sharp, as Viehweg-Zuo \cite{VZ01}*{Theorem 0.2} suggests. Theorem \ref{thm: degree of discriminant locus} and Corollary \ref{cor: degree of discriminant locus} give a sharp inequality whose proof is provided in the next section, which primarily uses the logarithmic base change theorem.

\section{Proofs}
\label{sec: proofs}

\subsection{Proof of Theorem \ref{thm: log general type descends}}

Theorem \ref{thm: log general type descends} is a generalization of Popa-Schnell \cite{PS22}*{Theorem A}, applying the logarithmic fiber product trick to their proof.

To begin with, we compactify \(X\) and \(Y\) using Hironaka's resolution of singularities. Therefore, we may assume \(f:(X,E)\to (Y,D)\) to be a morphism of log smooth pairs with \(E=f^{-1}(D)\) such that \(f|_{X\setminus E}:X\setminus E\to Y\setminus D\) is the initial smooth morphism of quasi-projective varieties.

The forward implication is immediate due to Maehara \cite{Maehara86}*{Corollary 2}: when \(Y\setminus D\) is of log general type, then \(\bar\k(X\setminus E)= \k(F)+\dim Y\). It remains to prove the converse implication, assuming \(\bar\k(X\setminus E)= \k(F)+\dim Y\). Recall the Easy Addition formula \cite{Mori87}*{Corollary 1.7}, applied to the Cartier divisor \(K_X+E\):
\[
\bar\k(X\setminus E)=\k(X,K_X+E)\le \k(F)+\dim Y.
\]
The equality holds if and only if there exists a positive integer \(N\) and an ample divisor \(A\) on \(Y\) such that
\[
f^*A\subset \w_X(E)^{\tensor N},
\]
by \cite{Mori87}*{Proposition 1.14}. This is an adjoint of the inclusion \(A\subset f_*(\w_X(E)^{\tensor N})\). Applying Corollary \ref{cor: logarithmic fiber product trick}, we obtain the following inclusions:
\begin{align*}
\left(A\tensor \left(\w_Y(D)^{\tensor -N}\right)\right)^{\tensor N}
&\subset\left[\bigotimes^{N} f_* \left(\w_X(E)/\w_Y(D)^{\tensor N}\right)\right]^{**}\\
&\subset\left[f^{(N)}_*\left(\w_{X^{(N)}}(E^{(N)})/\w_Y(D)^{\tensor N}\right)\right]^{**}.
\end{align*}
Therefore, taking \(\L=A\tensor (\w_Y(D)^{\tensor -N})\), there exists a logarithmic Higgs sheaf \(\F_\bullet\) with poles along \(D\) which satisfies the properties of Theorem \ref{thm: construction of logarithmic Higgs sheaves}, by Remark \ref{rmk: assumption in logarithmic Higgs sheaves}. From Lemma \ref{lem: CP setup}, we have a pseudo-effective line bundle \(P\) and a nonzero morphism
\[
A^{\tensor r}\tensor P\to (\Omega_Y(\log D))^{\tensor kr}\tensor (\w_Y(D))^{\tensor Nr}
\]
for some \(r>0\), \(k\ge 0\). By Theorem \ref{thm: CP pseudoeffective}, \(K_Y+D\) is pseudo-effective. Note that \(\w_Y(D)\) is a sub-line bundle in \((\Omega_Y(\log D))^{\tensor \dim Y}\). Hence, there exists some positive integer \(N'\) such that 
\[
A^{\tensor r}\tensor P\to(\Omega_Y(\log D))^{\tensor N'}
\]
is a nonzero morphism. Thus, Theorem \ref{thm: CP main} implies that \(K_Y+D\) is the sum of a big divisor and a pseudo-effective divisor, so it is big as desired.

\subsection{Proof of Theorem \ref{thm: effectivity descends to pseudo-effectivity} (1)}

We use a similar technique as in the proof of Theorem \ref{thm: log general type descends}. Since \(K_X+(1-\epsilon)E\) is effective, there exists a large enough, sufficiently divisible positive integer \(N\) such that
\[
f^*\O_Y(D)\subset \w_X(E)^{\tensor N}.
\]
This is an adjoint of the inclusion \(\O_Y(D)\subset f_*(\w_X(E)^{\tensor N})\). Applying Corollary \ref{cor: logarithmic fiber product trick} as in the proof of Theorem \ref{thm: log general type descends}, we get
\[
\left(\O_Y(D)\tensor \left(\w_Y(D)^{\tensor -N}\right)\right)^{\tensor N}
\subset\left[f^{(N)}_*\left(\w_{X^{(N)}}(E^{(N)})/\w_Y(D)^{\tensor N}\right)\right]^{**},
\]
and obtain a nonzero morphism
\[
\O_Y(D)^{\tensor r}\tensor P\to (\Omega_Y(\log D))^{\tensor kr}\tensor (\w_Y(D))^{\tensor Nr}
\]
induced by the logarithmic Higgs sheaf construction with poles along \(D\). Here, \(r>0\), \(k\ge 0\), and \(P\) is a pseudo-effective line bundle. By Theorem \ref{thm: CP pseudoeffective}, \(K_Y+D\) is pseudo-effective. From the split inclusion \(\w_Y(D)\subset (\Omega_Y(\log D))^{\tensor \dim Y}\), there exists some positive integer \(N'\) such that
\[
\O_Y(D)^{\tensor r}\tensor P\to(\Omega_Y(\log D))^{\tensor N'}
\]
is a nonzero morphism. If \(Q\) is its cokernel, then \(c_1(Q)\) is pseudo-effective by Theorem \ref{thm: CP main}. Since
\begin{align*}
c_1((\Omega_Y(\log D))^{\tensor N'})&=N'(\dim Y)^{N'-1}c_1(K_Y+D)\\
&=rc_1(D)+c_1(P)+c_1(Q)
\end{align*}
and \(P, Q\) are pseudo-effective, there exists \(\delta>0\) such that \(K_Y+(1-\delta)D\) is pseudo-effective.

The sum of a big divisor and a pseudo-effective divisor is big. If \(-K_Y+ND\) is big for some non-negative integer \(N\), then the sum
\[
(N+1-\delta)D=(-K_Y+ND)+(K_Y+(1-\delta)D)
\]
is big. Hence \(D\) is big, and therefore
\[
K_Y+D=(K_Y+(1-\delta)D)+\delta D
\]
is also big.

\subsection{Proof of Theorem \ref{thm: degree of discriminant locus}}
Let \(\mu:(\widetilde Y,\widetilde D)\to (Y,\Delta(f))\) be a log resolution of \((Y,\Delta(f))\). By taking a suitable resolution \(\widetilde X\) of the main component of \(X\times_Y \widetilde Y\), we obtain an induced morphism \(\tilde f:(\widetilde X,\widetilde E)\to (\widetilde Y,\widetilde D)\) of projective log smooth pairs with \(\widetilde E:=\tilde f^{-1}(\widetilde D)\), and \(\tilde f|_{\widetilde X\setminus \widetilde E}\) is smooth. By Theorem \ref{thm: log general type descends} and Theorem \ref{thm: effectivity descends to pseudo-effectivity} (1), \(K_{\widetilde Y}+\widetilde D\) is big in all cases, which implies that \(Y\setminus \Delta(f)\) is of log general type. Therefore, from Lemma \ref{lem: weakly positive and big sheaves} (1), \(K_Y+D\) is big.

\section{Popa's superadditivity of logarithmic Kodaira dimension}
\label{sec: additivity of logarithmic Kodaira dimension}

Popa conjectured recently that for a smooth projective morphism with connected fibers, the (logarithmic) Kodaira dimension is additive.

\begin{conj}[\cite{Popa21}*{Conjecture 3.1}]
\label{conj: additivity of logarithmic Kodaira dimension}
Let \(f:X\to Y\) be a smooth projective morphism with connected fibers between smooth quasi-projective varieties. Then
\[
\bar\k(X)=\bar\k(Y)+\k(F)
\]
where \(F\) is the general fiber of \(f\).
\end{conj}

As mentioned after Corollary \ref{cor: superadditivity of logarithmic Kodaira dimension}, the subadditivity \(\bar\k(X)\ge\bar\k(Y)+\k(F)\) is implied by the logarithmic Iitaka conjecture. Hence, Popa's conjecture suggests its counterpart, namely the superadditivity \(\bar\k(X)\le\bar\k(Y)+\k(F)\), when the morphism is smooth.

As a preparation, recall that the numerical log Kodaira dimension of \(Y\) is defined as
\[
\bar\k_{\sigma}(Y):=\k_\sigma(\bar Y,K_{\bar Y}+D),
\]
where \(\bar Y\) is a compactification of \(Y\) with boundary a reduced simple normal crossing divisor \(D\) and \(\k_\sigma\) is Nakayama's numerical dimension \cite{Nakayama04}*{V.2.5 Definition}. This is well defined, independent of the choice of a compactification. Additionally, it is well known that if \(\bar \k_\sigma(Y)=0\), then \(\bar \k(Y)=0\) (see Kawamata \cite{Kawamata13}*{Theorem 1} combined with Nakayama \cite{Nakayama04}*{V.2.7 Proposition (8)}).

We first prove superadditivity under the additional assumption that the numerical log Kodaira dimension of the base \(Y\) is equal to zero.

\begin{prop}
\label{prop: superadditivity numerical dimension zero}
With the notation in Conjecture \ref{conj: additivity of logarithmic Kodaira dimension}, assume that \(\bar\k_{\sigma}(Y)=0\). Then
\[
\bar \k(X)\le\k(F).
\]
\end{prop}

The proof of this proposition uses the logarithmic base change theorem as in the proofs in Section \ref{sec: proofs}, with the following lemma. This is essentially \cite{PS22}*{Lemma 4.2}, modified for our use in the case of torsion free sheaves. We include its proof for completeness.

\begin{lem}
\label{lem: producing 2 sections}
Let \(\F\) be a torsion free sheaf on a normal projective variety \(Z\), globally generated on the complement of a codimension 2 locus. If \(h^0(Z,\F)>\mathrm{rank}\, \F\), then \(h^0(Z,\widehat\det \F)\ge 2\).
\end{lem}

\begin{proof}
Denote \(s:=\mathrm{rank}\, \F\). Via taking the double dual, we may assume that \(\F\) is reflexive. Let \(V\subset Z\) be the complement of a codimension 2 subvariety in \(Z\), over which \(\F\) is a globally generated locally free sheaf of rank \(s\). Consequently, there exists a generically isomorphic morphism
\[
\alpha_1:\O_{Z}^{\oplus s}\to \F.
\]
If \(\alpha_1\) is an isomorphism on \(V\), then the reflexive hull of \(\alpha_1\) is an isomorphism on \(Z\), which contradicts \(h^0(Z,\F)>s\). Hence, there exists a point \(z\in V\) such that \(\alpha_1|_z\) is not surjective as a linear map of vector spaces. Since \(\F\) is globally generated at \(z\), there exists a morphism
\[
\alpha_2:\O_{Z}^{\oplus s}\to \F
\]
such that \(\alpha_2|_z\) is an isomorphism of vector spaces. Then we have the following two nonzero global sections
\[
\widehat\det(\alpha_1):\O_{Z}\to \widehat\det\F, \quad \widehat\det (\alpha_2):\O_{Z}\to \widehat\det\F,
\]
such that \(\widehat\det (\alpha_1)\) vanishes at \(z\) and \(\widehat\det (\alpha_2)\) is an isomorphism at \(z\). Therefore, we have the conclusion.
\end{proof}

\begin{proof}[Proof of Proposition \ref{prop: superadditivity numerical dimension zero}]
We argue by contradiction: assume that \(\bar \k(X)>\k(F)\). As at the beginning of the proof of Theorem \ref{thm: log general type descends}, let \(f:(\bar X,E)\to (\bar Y,D)\) be a morphism of log smooth pairs with \(E=f^{-1}(D)\) such that \(f|_{\bar X\setminus E}:\bar X\setminus E\to \bar Y\setminus D\) is the initial smooth morphism. Notice that
\[
\mathrm{rank}\, \left(f_*\left(\w_{\bar X}(E)^{\tensor N}\right)\right)=P_N(F)
\]
where \(P_N(F)\) is the \(N\)-plurigenus of \(F\). Therefore, there exists a positive integer \(N\) such that
\[
h^0\left(\bar Y,f_*\left(\w_{\bar X}(E)^{\tensor N}\right)\right) > \mathrm{rank}\, \left(f_*\left(\w_{\bar X}(E)^{\tensor N}\right)\right).
\]

Let \(\F\) be the subsheaf of \(f_*\left(\w_{\bar X}(E)^{\tensor N}\right)\), generated by its global sections. Therefore, \(\F\) is torsion free and \(h^0(\bar Y,\F)>s:=\mathrm{rank}\, \F\). From the following inclusions
\[
\widehat\det\F\subset\left[\bigwedge^sf_*\left(\w_{\bar X}(E)^{\tensor N}\right)\right]^{**}\subset \left[\bigotimes^sf_*\left(\w_{\bar X}(E)^{\tensor N}\right)\right]^{**},
\]
we have 
\begin{align*}
\left(\widehat\det\F\tensor \left(\w_{\bar Y}(D)^{\tensor -Ns}\right)\right)^{\tensor N}
&\subset\left[\bigotimes^{Ns} f_* \left(\w_{\bar X}(E)/\w_{\bar Y}(D)^{\tensor N}\right)\right]^{**}\\
&\subset\left[f^{(Ns)}_*\left(\w_{\bar X^{(Ns)}}(E^{(Ns)})/\w_{\bar Y}(D)^{\tensor N}\right)\right]^{**}
\end{align*}
by Corollary \ref{cor: logarithmic fiber product trick}. As in the proofs of Theorem \ref{thm: log general type descends} and Theorem \ref{thm: effectivity descends to pseudo-effectivity} (1), we obtain a nonzero morphism
\[
A^{\tensor r}\tensor P\to (\Omega_{\bar Y}(\log D))^{\tensor kr}\tensor (\w_{\bar Y}(D))^{\tensor Nsr},
\]
where \(A:=\widehat\det\F\), \(r>0\), \(k\ge 0\), and \(P\) is a pseudo-effective line bundle. Hence, there exists some positive integer \(N'\) and an injection
\[
A^{\tensor r}\tensor P\to(\Omega_{\bar Y}(\log D))^{\tensor N'}
\]
with the quotient \(Q\), which implies that
\[
N'(\dim \bar Y)^{N'-1}c_1(K_{\bar Y}+D)=rc_1(A)+c_1(P)+c_1(Q).
\]
We now take Nakayama's numerical dimension on both sides. Due to Nakayama \cite{Nakayama04}*{V.2.7 Proposition (1)} and Theorem \ref{thm: CP main}, we have
\[
\bar \k_\sigma(Y)=\k_\sigma(\bar Y,K_{\bar Y}+D)\ge \k_\sigma(\bar Y,A).
\]
However by Lemma \ref{lem: producing 2 sections}, we have \(\k_\sigma(\bar Y,A)\ge 1\), which contradicts \(\bar \k_\sigma(Y)=0\).
\end{proof}

Theorem \ref{thm: superadditivity of logarithmic Kodaira dimension} states the superadditivity of the log Kodaira dimension when the very general fiber of the log Iitaka fibration of the base has a good minimal model. This is obtained as a consequence of Proposition \ref{prop: superadditivity numerical dimension zero} via the Easy addition formula.

\begin{proof}[Proof of Theorem \ref{thm: superadditivity of logarithmic Kodaira dimension}]
Let \(\mu:(\widetilde Y,\widetilde D)\to (\bar Y,D)\) be a log resolution  with the log Iitaka fibration \(\eta:(\widetilde Y, \widetilde D)\to I\). Since the log Kodaira dimension is invariant under birational modifications (see e.g. \cite{Fujino20}*{Lemma 2.3.34}), we have \(\bar \k(\widetilde Y\setminus \widetilde D)=\bar \k(Y)\) and \(\bar \k(X\times_Y\widetilde Y\setminus \widetilde D)=\bar\k(X)\). Hence, taking the base change of \(f\) via \(\mu\), we may assume \((\bar Y,D)=(\widetilde Y, \widetilde D)\).

Let \(f:(\bar X,E)\to (\bar Y,D)\) be a morphism of log smooth pairs with \(E=f^{-1}(D)\) such that \(f|_{\bar X\setminus E}:\bar X\setminus E\to \bar Y\setminus D\) is the initial smooth morphism. Then we have the following composition of morphisms
\[
(\bar X, E)\xrightarrow{f} (\bar Y, D)\xrightarrow{\eta} I. 
\]
Denote by \((G, D|_G)\) (resp. \((H,E|_H)\)) the very general fiber of \(\eta\) (resp. \(\eta\circ f\)). Notice that the restriction map
\[
f|_{H}: (H,E|_H)\to (G, D|_G)
\]
is a morphism of log smooth pairs with \(E|_H=f|_H^{-1}(D|_G)\), and \(f|_{H\setminus E|_H}\) is smooth. From the assumption, the pair \((G, D|_G)\) has a good minimal model of log Kodaira dimension zero, which implies that
\[
\k_\sigma(G,K_G+D|_G)=0.
\]
By Proposition \ref{prop: superadditivity numerical dimension zero}, we have
\[
\k(H,K_H+E|_H)\le \k(F),
\]
hence, by the Easy addition formula applied to \(\eta\circ f\), we obtain
\begin{align*}
\bar\k(X)=\k(\bar X,K_{\bar X}+E)&\le \k(H,K_H+E|_H)+\dim I\\
&\le \k(F)+\bar\k(Y)
\end{align*}
which concludes the proof.
\end{proof}

In the remaining section, we state the logarithmic Iitaka conjecture for projective morphisms over quasi-projective curves. This is known to hold for experts, but not explicitly stated in the literature, so we include a short proof.

\begin{prop}
\label{prop: log Iitaka conjecture over curve}
Let \(f:X\to Y\) be a projective morphism with connected fibers between smooth quasi-projective varieties. If \(\dim Y=1\), then
\[
\bar \k(X)\ge\bar \k(Y)+\k(F).
\]
\end{prop}

\begin{proof}
When \(Y\) is of log general type, the conclusion follows from Maehara \cite{Maehara86}*{Corollary 2}. When \(Y\) is a projective curve of genus \(1\), it follows from Kawamata \cite{Kawamata82}*{Theorem 2}. When \(Y=\P^1\) or \(Y=\P^1-\{\infty\}\), we have \(\bar \k(Y)=-\infty\).

It suffices to prove \(\bar \k(X)\ge\k(F)\) when \(Y=\P^1-\{0,\infty\}\). To begin with, we compactify the morphism \(f\) as \(\bar f:(\bar X,E)\to (\P^1, D)\) so that \(D=0+\infty\), \(E=\bar f^{-1}(D)\), and \(\bar f|_{\bar X\setminus E}=f\). As in the proof of Viehweg-Zuo \cite{VZ01}, there exists a cyclic cover \(\P^1\to \P^1\) of degree \(d\), \'etale over \(\C^*=\P^1 - \{0,\infty\}\), which induces a semistable reduction at \(\{0,\infty\}\) for a sufficiently large and divisible \(d\) (Kempf et al. \cite{KKMS}). Let the resulting semistable reduction at \(\{0,\infty\}\) be the following commutative diagram:
\begin{displaymath}
\xymatrix{
{\bar X'}\ar[r]^-{\nu}\ar[d]_-{\bar f'}& {\bar X}\ar[d]_-{\bar f}\\
{\P^1}\ar[r]^{z\mapsto z^d}&{\P^1}
}
\end{displaymath}
Denote \(X':=\nu^{-1}(X)\). Since \(\nu|_{X'}:X'\to X\) is a finite \'etale morphism, we have \(\bar \k(X')=\bar\k(X)\) by Iitaka \cite{Iitaka77}*{Theorem 3}. Therefore, we reduce to the case where the fibers of \(\bar f\) over \(\{0,\infty\}\) are reduced simple normal crossing divisors.

Since \(\O_{\bar X}(E)\isom\bar f^*\O_{\P^1}(2)\isom\bar f^*\w_{\P^1}^{-1}\), we have \(\w_{\bar X}(E)\isom\w_{\bar X/\P^1}\). By Viehweg \cite{Viehweg83}*{Theorem III}, \(\bar f_* \w_{\bar X/\P^1}^{k}\) is weakly positive for all \(k>0\). Therefore, the vector bundle \(\bar f_* \w_{\bar X/\P^1}^{k}\) decomposes into a direct sum of line bundles of non-negative degree on \(\P^1\), which implies that
\[
h^0(\P^1, \bar f_* \w_{\bar X/\P^1}^{k})\ge \mathrm{rank}\, \left(\bar f_* \w_{\bar X/\P^1}^{k}\right)=P_k(F)
\]
where \(P_k(F)\) is the \(k\)-plurigenus of \(F\). Hence, \(h^0(\bar X, k(K_{\bar X}+E))\ge P_k(F)\) for all \(k>0\), so we have \(\bar \k(X)\ge\k(F)\).
\end{proof}

In particular, Theorem \ref{thm: superadditivity of logarithmic Kodaira dimension} and Proposition \ref{prop: log Iitaka conjecture over curve} imply Corollary \ref{cor: superadditivity of logarithmic Kodaira dimension}.

\section{Uniruledness of fibrations over projective spaces}
\label{sec: uniruledness}

Assume we have a surjective projective morphism \(f:X\to \P^1\) with at most two singular fibers. From Viehweg-Zuo \cite{VZ01}, we have \(\k(X)=-\infty\), and the non-vanishing conjecture implies that \(X\) is uniruled. In other words, putting these two together, we have:

\begin{conj}
\label{conj: 2 singular fiber uniruled}
Let \(X\) be a smooth projective variety, and \(f:X\to \P^1\) be a surjective morphism with at most \(2\) singular fibers. Then \(X\) is uniruled.
\end{conj}

As mentioned in the introduction after Theorem \ref{thm: 2 singular fiber uniruled}, Pieloch \cite{Pieloch21} proved the conjecture when \(f\) has at most one singular fiber, and obtained a partial result when \(f\) has two singular fibers. The proofs are symplectic; the algebro-geometric proofs were previously unknown.

In what follows, we give an algebraic proof of this conjecture under the additional assumption that the general fiber has a good minimal model. The key idea is the invariance of the canonical rings of smooth fibers stated in the following theorem. This contains Theorem \ref{thm: 2 singular fiber uniruled}. We also state an analogous conjecture over a higher dimensional base below.

\begin{thm}
\label{thm: uniruledness}
In the setting of Conjecture \ref{conj: 2 singular fiber uniruled}, assume that the fibers of \(f\) are connected. Then, the canonical rings of all smooth fibers are isomorphic.\footnote{In particular, if the fibers are of general type, then \(f\) is birationally isotrivial.} In addition, if the general fiber \(F\) has a good minimal model, then \(X\) is uniruled.
\end{thm}

\begin{proof}
As in the proof of Proposition \ref{prop: log Iitaka conjecture over curve}, it suffices to consider the case when \(f\) is semistable with at most \(2\) singular fibers over \(\{0,\infty\}\). Let \(D=0+\infty\) and \(E=f^{-1}(D)\). Since \(K_{\P^1}+(1-\delta)D\) is not pseudo-effective for all \(\delta>0\), \(K_{X}+(1-\epsilon)E\) is not effective for all \(\epsilon>0\) by Theorem \ref{thm: effectivity descends to pseudo-effectivity} (1). Therefore, there is no nonzero morphism
\[
\O_{\P^1}(1)\to f_* \w_{ X/\P^1}^{k},
\]
which implies that
\begin{equation}
\label{eqn: 2 semistable fiber}
f_* \w_{ X/\P^1}^{k}\isom\bigoplus^{P_k(F)}\O_{\P^1}.
\end{equation}
Alternatively, \cite{VZ01}*{Proposition 4.2} proves that the degree of \(f_* \w_{ X/\P^1}^{k}\) is equal to zero, which implies \eqref{eqn: 2 semistable fiber}.
As a consequence, the multiplication map
\[
f_* \w_{ X/\P^1}^{k}\tensor  f_* \w_{X/\P^1}^{l}\to f_* \w_{ X/\P^1}^{k+l}
\]
is constant on the fibers, and hence is isomorphic to
\[
\left[H^0(F,\w_F^k)\tensor_\C H^0(F,\w_F^l)\to H^0(F,\w_F^{k+l})\right] \tensor_\C \O_{\P^1}.
\]
Therefore, the canonical ring of every smooth fiber is isomorphic to the canonical ring \(R(F,\w_F)\) of \(F\), which is isomorphic to \(R(X, \w_{X/\P^1})\), establishing the first assertion. Accordingly, the Iitaka model \(I\) of \((X,\w_{X/\P^1})\) is the Iitaka model of \((F,\w_F)\). Resolving the indeterminacy of the Iitaka morphism, we obtain the following diagram:
\begin{displaymath}
\xymatrix{
{X'}\ar@/_/[dd]_{f'}\ar[d]^-{\mu}\ar[rd]^-{\eta}&\\
{X}\ar@{-->}[r]^-{}\ar[d]^-{f}& {I}\\
{\P^1}&
}
\end{displaymath}
Notice that \(\eta\) is also the Iitaka fibration of \((X',\w_{X'/\P^1})\), and the restriction map \(\eta|_{f'^{-1}(y)}:f'^{-1}(y)\to I\) for the general point \(y\in \P^1\) is the Iitaka fibration of \(F\).

Suppose \(F\) has a good minimal model. Then, it is well known that every fiber in the smooth locus of the Iitaka fibration of \(F\) has a good minimal model (see e.g. \cite{HMX18}*{Theorem 1.2}). Now, it suffices to prove that \(X'\) is uniruled, hence that the very general fiber \(Z'\) of \(\eta\) is uniruled. We have \(\k(Z',\w_{X'/\P^1}|_{Z'})=0\) by Iitaka's theory of D-dimension \cite{Mori87}*{(1.11)}, which implies that \(\k(Z',\w_{Z'/\P^1})=0\). Consider \(f':Z'\to \P^1\), then its general fiber \(G'\) is the smooth fiber of the Iitaka fibration of \(F\). By Viehweg's weak positivity theorem, we have
\[
f'_*\w_{Z'/\P^1}^k\isom\O_{\P^1},
\]
whenever \(\w_{G'}^k\) admits sections. Since \(G'\) admits a good minimal model from the assumption, it is well known that \(Z'\) has a relative good minimal model \(Z''\) over \(\P^1\). If we denote \(f'':Z''\to \P^1\), then \(\w_{Z''}\) is \(f''\)-semi-ample. From \(f''_*\w_{Z''/\P^1}^k\isom\O_{\P^1}\), we have \(\w_{Z''}^k\isom f''^*\w_{\P^1}^k\). Accordingly, the canonical divisor \(K_{Z''}\) is not pseudo-effective, which implies that any resolution of \(Z''\) is uniruled \cite{BDPP}. Therefore, \(Z'\) is uniruled as desired.
\end{proof}

\begin{rmk}
In fact, Theorem \ref{thm: uniruledness} extends to the case where \(f\) has disconnected fibers. Indeed, the Stein factorization of \(f:X\to \P^1\) is \(\P^1\), and the induced morphism again has at most two singular fibers. This can be checked from the observation that a finite covering of \(\P^1\) branched over at most two points is either an identity map or a cyclic map, i.e. \(z\mapsto z^d\) for \(d>1\) with an appropriate choice of the coordinate \(z\).
\end{rmk}

Combining Corollary \ref{cor: degree of discriminant locus} and the non-vanishing conjecture, we obtain the following higher dimensional analogue of Conjecture \ref{conj: 2 singular fiber uniruled}. In the remaining section, we prove this conjecture for \(n=2\) under the same additional assumption as in Theorem \ref{thm: uniruledness}, that the general fiber has a good minimal model.

\begin{conj}
\label{conj: higher dimensional uniruled}
Let \(X\) be a smooth projective variety, and \(f:X\to \P^n\) be a surjective morphism with either \(\dim \Delta(f)\le n-2\) or \(\deg \Delta(f)\le n+1\). Then \(X\) is uniruled.
\end{conj}

\begin{cor}
\label{cor: degree 3 uniruled}
In the setting of Conjecture \ref{conj: higher dimensional uniruled}, assume that \(n=2\) and the fibers of \(f\) are connected. If the general fiber \(F\) has a good minimal model, then \(X\) is uniruled.
\end{cor}

\begin{proof}
The discriminant locus \(\Delta(f)\) is the union of a plane curve \(C\) (possibly empty) of degree at most \(3\) and a finite set \(S\) of points.

\textit{Case 1.} \(\deg C\le 2\) or \(C=\phi\).

The general projective line in \(\P^2\) intersects \(\Delta(f)\) at at most \(2\) points. Then the preimage of the line is uniruled by Theorem \ref{thm: uniruledness}. Therefore, \(X\) is uniruled.

\textit{Case 2.} \(\deg C=3\).

If \(C\) is a smooth cubic plane curve, \(\P^2\) is covered by projective lines tangent to \(C\), and only finitely many such lines pass through \(S\). If \(C\) is a singular cubic plane curve, pick a singular point \(p\) and consider all projective lines through that point. Likewise, only finitely many such lines pass through \(S\). In all cases, \(\P^2\) is generically covered by lines that intersect \(\Delta(f)\) at at most \(2\) points. Therefore, we conclude that \(X\) is uniruled.
\end{proof}

\begin{rmk}
For \(n\ge 3\), the possible configurations for the discriminant locus in Conjecture \ref{conj: higher dimensional uniruled} are very complicated; the case by case analysis in the proof of Corollary \ref{cor: degree 3 uniruled} is no longer valid. Still, if the degree of the divisorial component of the discriminant locus is at most \(3\) and \(F\) has a good minimal model, then \(X\) is uniruled. This is because when we slice \(\P^n\) by a general projective plane \(\P^2\subset \P^n\), then the preimage is uniruled by Corollary \ref{cor: degree 3 uniruled}.
\end{rmk}

\section{Boundary examples and further remarks}
\label{sec: further remarks}

We present boundary examples showing that the inequality in Corollary \ref{cor: degree of discriminant locus} is sharp. Specifically, we demonstrate morphisms to \(\P^n\) such that the domains have the Kodaira dimension equal to zero and the discriminant loci are divisors of degree \(n+2\) in \(\P^n\).

\begin{ex}[Isotrivial example with discriminant locus of minimal degree]
\label{ex: boundary not connected fiber}
First, we construct a finite cover of \(\P^n\) with trivial canonical bundle branched over a hypersurface of degree \(n+2\).
Let \(D\subset \P^n\) be a degree \(n+2\) hypersurface with simple normal crossing singularities. Consider a degree \(n+2\) cyclic cover \(X\to \P^n\) associated to \(D\). Then \(X\) has Gorenstein rational singularities and \(\w_X\isom \O_X\). We resolve the singularities of \(X\) by a series of smooth blow-ups over the branch locus, \(\mu:\widetilde X\to X\). Then the composition \(f:\widetilde X\to \P^n\) is an example with \(h^0(\w_{\widetilde X})=1\) and the discriminant locus \(\Delta (f)=D\) is the hypersurface of degree \(n+2\).
\end{ex}

Now, we use the construction above to build an example with connected fibers over \(\P^n\), with \(n\ge 2\). To begin with, when \(D\) is a smooth degree \(n+2\) hypersurface, the cover \(f:X\to \P^n\) we constructed above is a smooth Galois cover with a Galois group \(G=\Z/(n+2)\Z\). Consider a \(G\)-action on \(X\times X\):
\[
\sigma:G\times X\times X\to X\times X, \quad \sigma(g,x_1,x_2):= (g^{-1}x_1,gx_2).
\]
Since \(G\) is abelian, the action is well defined. Hence, the projection onto the second factor \(pr_2:X\times X\to X\) is a \(G\)-equivariant morphism, from which we obtain an induced morphism of quotients by the \(G\)-action:
\[
\overline{pr_2}:(X\times X)/G\to X/G\isom \P^n.
\]
Notice that all the fibers of \(\overline{pr_2}\) over \(\P^n\setminus D\) are isomorphic to \(X\). Therefore, if we denote by \(\mu: W \to (X\times X)/G\) a resolution of singularities, which is an isomorphism away from the singular locus, then we have \(\k(W)=0\) from the following lemma, and the induced morphism \(h:W\to \P^n\) satisfies \(\Delta(h)=D\).

\begin{lem}
Under the above notation, \((X\times X)/G\) has canonical singularities and the Kodaira dimension of a desingularization of \((X\times X)/G\) is equal to zero.
\end{lem}

\begin{proof}
Let \(E=f^{-1}(D)\) be the ramification divisor of \(f\) on \(X\). The \(G\)-action is free outside of \(E\times E\subset X\times X\); hence, it suffices to prove that the quotient singularities at points \((x_1,x_2)\in E\times E\) are canonical. Let \(\zeta\) be a primitive \((n+2)\)-th root of unity. Locally analytically, the \(G\)-action near \(x\in E\subset X\) is equivalent to
\[
(z_1,\dots,z_{n-1}, z_n)\mapsto (z_1,\dots,z_{n-1}, \zeta^m z_n), \quad m\in \Z/(n+2)\Z
\]
where \(z_1,\dots,z_{n-1}, z_n\) are local coordinates and \(z_n=0\) is a local equation of \(E\). Therefore, at each point \((x_1,x_2)\in E\times E\), we obtain a singularity, which is analytically isomorphic to
\[
\C^{2n-2}\times \left(\C^2/\tfrac{1}{n+2}(1,q)\right)
\]
where \(\C^2/\tfrac{1}{n+2}(1,q)\) is the standard abbreviation of the surface cyclic quotient singularity of type \(\tfrac{1}{n+2}(1,q)\). This singularity is well known to be canonical if and only if \(q=-1\) (e.g. Reid-Tai Criterion \cite{Reid87}*{(4.11)}). Therefore, the locus in \(E\times E\) where the quotient singularities are canonical is both open and closed in the analytic topology. Notice that \(q=-1\) on the diagonal \(E\subset E\times E\) from the definition of the \(G\)-action. Thus, \((X\times X)/G\) has canonical singularities.

Since the quotient morphism \(X\times X\to (X\times X)/G\) is \'etale away from a codimension \(2\) subvariety, we have an equality
\[
\k(X\times X, \w_{X\times X})=\k\left((X\times X)/G, \w_{(X\times X)/G}\right)=0,
\]
which completes the proof.
\end{proof}

In particular, when \(n=1\), we have a triple cyclic cover \(f:C\to \P^1\), with \(C\) an elliptic curve branched over three points. Then the minimal resolution of \((C\times C)/G\) turns out to be an elliptic K3 surface fibered over \(\P^1\) with three singular fibers sitting above the branch points. See \cite{Sawon14}*{Section 2} for details.

The next example is a threefold with zero Kodaira dimension, mapping to \(\P^1\) with three singular fibers, obtained as an application of Viehweg's base change theorem. In this example, the fibers are connected and the family is birationally non-isotrivial, unlike in the previous one. We say that the family is birationally isotrivial if the general fibers are pairwise birationally isomorphic.

\begin{ex}[Non-isotrivial example with discriminant locus of minimal degree]
\label{ex: boundary connected fiber}
Let \(f:S\to \P^1\) be a family of elliptic curves parametrized by
\[
y^2=x(x-1)(x-\lambda).
\]
To be specific, let \(S\subset \P^2\times \P^1\) be a hypersurface of type \((3,1)\) defined by \(y^2z=x(x-z)(x-\lambda z)\), where \(x,y,z\) are the coordinates of \(\P^2\) and \(\lambda\) is the coordinate of \(\P^1=\A^1\cup \{\infty\}\). By adjunction, \(\w_S\isom f^*\O_{\P^1}(-1)\), so that
\[
f_*\w_{S/\P^1}\isom \O_{\P^1}(1).
\]
Notice that \(f\) has three singular fibers at \(\{0,1,\infty\}\); over \(\{0,1\}\) we have a nodal cubic curve, and over \(\{\infty\}\) we have a union of three concurrent lines. We can verify by computation that \(S\) has \(A_1\)-singularities at the nodes over \(\{0,1\}\) and has an \(A_4\)-singularity at the planar triple point over \(\{\infty\}\). Let \(\mu:\widetilde S\to S\) be the minimal resolution of singularities and \(\tilde f:\widetilde S\to \P^1\) be the induced morphism. Then \(\tilde f\) has semistable fibers over \(\{0,1\}\) and a non-semistable fiber over \(\{\infty\}\), with
\[
\tilde f_*\w_{\widetilde S/\P^1}\isom \O_{\P^1}(1).
\]
Take an automorphism of \(\P^1\), fixing \(0\) and flipping \(1\) and \(\infty\). Let \(\tilde f':\widetilde S\to \P^1\) be the induced family, which now has semistable fibers over \(\{0,\infty\}\) and a non-semistable fiber over \(\{1\}\). Let \(X:=\widetilde S\times_{\P^1}\widetilde S\) be the fiber product of \(\tilde f\) and \(\tilde f'\), and let \(\widetilde X\) be the resolution of \(X\), which is an isomorphism over the smooth locus of \(X\). Let \(h:\widetilde X\to \P^1\) be the induced morphism. By Viehweg's base change theorem \ref{thm: base change theorem},
\[
h_*\w_{\widetilde X/\P^1}\isom\O_{\P^1}(2),
\]
which implies that \(h^0(\w_{\widetilde X})=1\). In fact, we have \(\k(\widetilde X)= 0\), analyzing the pushforwards of pluri-canonical bundles. Therefore, there exists a threefold \(\widetilde X\) such that \(\k(\widetilde X)= 0\), and \(\widetilde X\to \P^1\) is birationally non-isotrivial with exactly three singular fibers.
\end{ex}

We close by addressing a question, raised by Kov\'acs \cite{Kovacs00}*{Question 0.6}, whether imposing a stronger condition on the Kodaira dimension of the total space increases the degree of the discriminant locus. We observe that this is not the case, by demonstrating morphisms \(f:X\to \P^n\) such that \(X\) is of general type and the degree of the discriminant locus \(\Delta(f)\) is either \(n+2\) or \(n+3\).

\begin{ex}[General type example with discriminant locus of small degree]
We first explain a general strategy to construct an explicit example. Suppose we have a smooth Galois cover \(g:Y\to \P^n\), with a Galois group \(G\), branched over a smooth hypersurface \(D\) as in Example \ref{ex: boundary not connected fiber}. For a positive integer \(N\), the projection onto the N-th factor,
\[
pr:Y^N\to Y,
\]
is a \(G\)-equivariant morphism, where \(Y^N\) is endowed with the diagonal \(G\)-action. Then the induced morphism
\[
\overline{pr}:Y^N/G\to Y/G\isom \P^n
\]
has \(\Delta(\overline{pr})=D\), and the general fiber is connected. Furthermore, if \(Y\) is of general type, then a desingularization of \(Y^N/G\) is of general type for \(N\gg0\), by Caporaso-Harris-Mazur \cite{CHM}*{Corollary 4.1}. Therefore, we obtain a morphism \(f:X\to \P^n\) such that \(X\) is of general type and \(\Delta(f)=D\).

When \(n=1\), we start with a Galois Bely\u\i{} map \(C\to \P^1\) such that \(C\) is a curve of genus at least \(2\) (e.g. \cite{SV14}*{Section 6}). Then the construction above shows that the minimal degree of \(\Delta(f)\) is indeed \(3\). 

When \(n\ge 2\), take a cyclic cover \(g:Y\to \P^n\) branched over a smooth hypersurface of degree \(n+3\). Then \(Y\) is of general type; consequently, we have an example with \(\deg \Delta(f)=n+3\) via the construction above. We expect there to be an example with \(\deg \Delta(f)=n+2\), which we do not know yet.
\end{ex}

\end{document}